\documentclass[a4paper,12pt]{article}
\usepackage[left=2cm,right=2cm, top=2cm,bottom=3cm,bindingoffset=0cm]{geometry}

\usepackage{verbatim}
\usepackage{amsmath}
\usepackage{amsthm}
\usepackage{amssymb}
\usepackage{delarray}
\usepackage{cite}
\usepackage{hyperref}
\usepackage{mathrsfs}
\usepackage{tikz}
\usetikzlibrary{patterns}
\usepackage{caption}
\newcommand{\la}{\lambda}
\newcommand{\ga}{\gamma}

\DeclareCaptionLabelSeparator{dot}{. }
\theoremstyle{plain}

\numberwithin{equation}{section}

\newtheorem{thm}{Theorem}[section]
\newtheorem{lem}[thm]{Lemma}
\newtheorem{prop}[thm]{Proposition}

\theoremstyle{definition}

\newtheorem{ip}[thm]{Inverse Problem}

\theoremstyle{remark}

\newtheorem{remark}[thm]{Remark}

 \allowdisplaybreaks

\begin{document}
\begin{center}
{\large\bf Local solvability and stability for the inverse Sturm-Liouville \\[0.2cm] problem with polynomials in the boundary conditions}
\\[0.2cm]
{\bf Egor E. Chitorkin*, Natalia P. Bondarenko} \\[0.2cm]
\end{center}

\vspace{0.5cm}

{\bf Abstract.}  In this paper, we for the first time prove local solvability and stability of the inverse Sturm-Liouville problem with complex-valued singular potential and with polynomials of the spectral parameter in the boundary conditions. The proof method is constructive. It is based on the reduction of the inverse problem to a linear equation in the Banach space of bounded infinite sequences. We prove that, under a small perturbation of the spectral data, the main equation of the inverse problem remains uniquely solvable. Furthermore, we derive new reconstruction formulas for obtaining the problem coefficients from the solution of the main equation and get stability estimates for the recovered coefficients. 

\medskip


\medskip


\vspace{1cm}

\section{Introduction} \label{sec:intro}

This paper deals with the Sturm-Liouville problem
\begin{gather} \label{eqv}
-y'' + q(x) y = \lambda y, \quad x \in (0, \pi), \\ \label{bc}
y^{[1]}(0) = 0, \quad r_1(\la)y^{[1]}(\pi) + r_2(\la)y(\pi) = 0,
\end{gather}
where $q(x)$ is a complex-valued distribution potential of the class $W_2^{-1}(0,\pi)$, $\la$ is the spectral parameter, $r_1(\la)$ and $r_2(\la)$ are relatively prime polynomials, $y^{[1]} := y' - \sigma(x) y$ is the so-called quasi-derivative, $\sigma(x)$ is an antiderivative of $q(x)$, that is, $q = \sigma'$ in the sense of generalized functions, $\sigma \in L_2(0,\pi)$. We understand equation \eqref{eqv} in the following equivalent sense:
\begin{equation} \label{eqsi}
 -(y^{[1]})' - \sigma(x)y^{[1]} - \sigma^2(x)y = \la y, \quad x \in (0,\pi),
\end{equation}
and consider its solutions in the class $y, y^{[1]} \in AC[0,\pi]$.

We study the inverse problem that consists in the recovery of $\sigma(x)$, $r_1(\lambda)$, and $r_2(\lambda)$ from the spectral data. This paper aims to prove local solvability and stability of the inverse problem.
To the best of the authors' knowledge, these issues have not been investigated before for the Sturm-Liouville problems containing the spectral parameter in the boundary conditions.


Inverse problems of spectral analysis consist in the recovery of a given differential operator (e.g. the Sturm-Liouville operator) from some information about its spectrum. Such kind of problems have applications in quantum mechanics, chemistry, electronics, and other branches of science and engineering. Inverse spectral theory for the Sturm-Liouville operators with \textit{constant} coefficients in the boundary conditions has been developed fairly completely (see the monographs \cite{Mar77, Lev84, FY01, Krav20}). There is also a number of studies concerning eigenvalue problems with \textit{polynomial} dependence on the spectral parameter in the boundary conditions. Such kind of eigenvalue problems arise in physical applications, e.g., in mechanical engineering, diffusion, and electric circuit problems (see \cite{Ful77, Ful80} and references therein). 

\textit{Inverse} Sturm-Liouville problems with polynomial dependence on spectral parameter in the boundary conditions have been studied in \cite{Chu01, BindBr021, BindBr022, BindBr04, ChFr, FrYu, FrYu12, Wang12, YangXu15, YangWei18, Gul19, Gul20-ann, Gul20-ams, Gul23, Chit} and other papers. Some part of them deal with self-adjoint problems containing rational Herglotz-Nevanlinna functions of the spectral parameter in the boundary conditions (see, e.g., \cite{BindBr021, BindBr022, YangWei18, Gul19, Gul20-ann, Gul20-ams, Gul23}) that can be reduced to the problem with polynomial dependence on the spectral parameter. In contrast to the mentioned studies, this paper deals with the general non-self-adjoint problem.

A constructive solution of the inverse Sturm-Liouville problem on a finite interval with the polynomial boundary conditions in the general non-self-adjoint case has been obtained by Freiling and Yurko \cite{FrYu} by using the method of spectral mappings. We note that those results are concerned only with regular potentials of the class $L_2(0,1)$.

In recent years, spectral analysis of differential operators with singular coefficients, which are the so-called distributional coefficients, has attracted much attention of mathematicians (see \cite{Gul19, SavShkal03, Hry03, Hry04, Sav05, Sav10, FrIgYu, Dj, BondTamkang, ChitBond}). Some properties of spectrum and solutions of differential equations with singular coefficients were studied in such papers as \cite{SavShkal03, Hry04, Sav05, Sav10}. The most complete results in the inverse problem theory have been obtained for the Sturm-Liouville operators with singular potentials (see \cite{Hry03, Hry04, Sav05, Sav10, FrIgYu, Dj, BondTamkang}). In particular, Hryniv and Mykytyuk \cite{Hry03, Hry04} generalized a number of classical results to the case of $W_2^{-1}$-potentials, by using the transformation operator method. Savchuk and Shkalikov \cite{Sav10} proved the uniform stability of the inverse Sturm-Liouville problems with potentials of $W_2^{\alpha}$, where $\alpha > -1$.  There were also studies of inverse problems with singular potentials for the quadratic Sturm-Liouville pencils (see \cite{Hry12, Pr13, BondGaidel, Kuz23}) and for the matrix Sturm-Liouville operators (see \cite{Myk09, Eck15, BondAMP21}).

The method of spectral mappings has been transferred to the Sturm-Liouville operators with potentials of the class $W_2^{-1}$ in \cite{FrIgYu, BondTamkang, ChitBond, Yur02}. In particular, in \cite{ChitBond}, we have obtained a constructive solution of an inverse problem for the Sturm-Liouville equation \eqref{eqv} with singular potential and the polynomial boundary conditions \eqref{bc}. Namely, the inverse problem has been reduced to a linear equation in the Banach space of bounded infinite sequences and reconstruction formulas for the coefficients $\sigma(x)$, $r_1(\la)$, and $r_2(\la)$ have been derived. The developed algorithm has opened a perspective of studying solvability and stability for the inverse problem.


Local solvability and stability of inverse problems for various classes of the Sturm-Liouville operators were investigated in \cite{FY01, Borg46, McL88, Hry03, HK11, BondButTr17, BK19, Bond20, GuoMa23, Yur02} and other studies.
Local solvability is an important property of inverse problems, especially in the cases when it is difficult to prove the global solvability.
Stability is significant for justification of numerical methods. To the best of the authors' knowledge, stability of inverse problems for differential operators with polynomial dependence on the spectral parameter in the boundary conditions has not been investigated before. Regarding solvability of such inverse problems, there are the results of Guliyev \cite{Gul19, Gul20-ann, Gul20-ams} for the self-adjoint case. However, for the non-self-adjoint case, there were no studies concerning solvability of inverse problems with polynomials in the boundary conditions. Thus, the present paper aims to fill this gap.

Let us formulate the main results of this paper.
Denote by $L = L(\sigma, r_1, r_2)$ the boundary value problem \eqref{eqv}-\eqref{bc}. Obviously, the polynomials of the boundary conditions can be represented in the form
$$
r_1(\la) = 	\sum \limits_{n=0}^{M_1} c_n \la^n, \quad  r_2(\la) = \sum \limits_{n=0}^{M_2} d_n \la^n, \quad M_1, M_2 \ge 0.
$$

We will write $(r_1, r_2)\in\mathcal R$ if $M_1= M_2$ and $c_{M_1} = 1$. Note that the higher coefficients $d_{M_2}$, $d_{M_1}$, \dots  of $r_2(\lambda)$ can be equal zero. Throughout this paper, we assume that $(r_1, r_2)\in\mathcal R$. The opposite case $M_2 < M_1$ can be considered similarly.

It has been proved in \cite{ChitBond} that the spectrum of the boundary value problem $L$ is a countable set of eigenvalues. They can be numbered as $\{\la_n\}_{n \ge 1}$ according to their asymptotic behavior:
\begin{gather}\label{la_asymp}
\rho_n = \sqrt{\la_n} = n - M_1 - 1 + \varkappa_n, \quad \{ \varkappa_n\} \in l_2.
\end{gather}

In this paper, we assume that the eigenvalues are simple. Then, the Weyl function $M(\la)$, which will be defined in Section~\ref{sec:prelim}, has simple poles at $\la = \la_n$, $n \ge 1$, and the weight numbers are defined as $\alpha_n = \mathop{Res}_{\la = \la_n} M(\la)$, $n \ge 1$. They fulfill the asymptotics (see \cite{ChitBond}):
\begin{gather}\label{alpha_asymp}
\alpha_n = \dfrac{2}{\pi} + \varkappa_n^0, \quad \{ \varkappa_n^0\} \in l_2.
\end{gather}

We call the numbers $\{\la_n, \alpha_n\}_{n \ge 1}$ \textit{the spectral data} of the problem $L$ and study the following inverse problem.

\begin{ip}\label{ip:main3}
	Given the spectral data $\{ \la_n, \alpha_n \}_{n \ge 1}$, find $\sigma(x)$, $r_1(\la)$, and $r_2(\la)$.
\end{ip}

Along with $L$, consider another boundary value problem $\tilde L = L(\tilde \sigma, \tilde r_1, \tilde r_2)$ of the same form but with different coefficients. We agree that, if a symbol $\ga$ denotes an object related to $L$, then the symbol $\tilde \ga$ with tilde will denote the analogous object related to $\tilde L$. Our main result is the following theorem on local solvability and stability of Inverse Problem~\ref{ip:main3}.

\begin{thm} \label{stability_thm}
	Let $\tilde L = L(\tilde \sigma, \tilde r_1, \tilde r_2)$ be a fixed boundary value problem of form \eqref{eqv}-\eqref{bc} with $\tilde \sigma \in L_2(0,\pi)$, $(\tilde r_1, \tilde r_2) \in \mathcal R$, and simple eigenvalues $\{ \tilde \la_n\}_{n \ge 1}$. 
	Then, there exists $\delta_0 > 0$, depending on problem $\tilde L$, such that, for any complex numbers $\{\la_n, \alpha_n\}_{n \ge 1}$ satisfying the condition
	\begin{equation} \label{estde}
	\delta := \bigg(\sum\limits_{n=1}^\infty(|\tilde \rho_n - \rho_n|+|\tilde\alpha_n - \alpha_n|)^2\bigg)^{\frac{1}{2}} \le \delta_0,
	\end{equation}
	where $\rho_n := \sqrt{\la_n}$, $\tilde \rho_n := \sqrt{\tilde \la_n}$,
	there exist a complex-valued function $\sigma(x) \in L_2(0, \pi)$ and polynomials $(r_1, r_2)\in\mathcal R$, which are the solution of Inverse Problem~\ref{ip:main3} for data $\{\la_n, \alpha_n\}_{n \ge 1}$. Moreover,
	\begin{equation} \label{stab}
	\|\sigma - \tilde\sigma\|_{L_2} \le C\delta, \quad |r_1(\la) - \tilde r_1(\la)| \le C\delta, \quad |r_2(\la) - \tilde r_2(\la)| \le C\delta
	\end{equation}
	for $\la$ on compact sets, where the constant $C$ depends only on $\tilde L$, $\delta_0$, and on a compact set in the estimates for $r_1(\la)$ and $r_2(\la)$.
\end{thm}

The proof of Theorem~\ref{stability_thm} is based on a constructive solution of the inverse Sturm-Liouville problem with polynomials in the boundary conditions with a non-zero model problem and ideas from \cite{BondTamkang, ChitBond} concerning with singular potentials. We consider the problem with polynomials only in the right boundary condition similarly to \cite{ChitBond} and derive new reconstruction formulas for an arbitrary model problem (see Theorem~\ref{mainthm}). So, our formulas generalize the reconstruction formulas from \cite{ChitBond}. In order to prove the correctness of the reconstruction formulas, we use the approximation approach. We show that the finite approximations imply polynomials of the required degrees and after that pass to limits. The most important step is the proof of the convergence for the series and infinite products in the reconstruction formulas. For $\sigma(x)$ and $r_1(\la)$, we use standard estimates. But for $r_2(\la)$ we have to apply the Cauchy criterion for the finite approximations. The described method allows us to construct the problem $L = L(\sigma, r_1, r_2)$ from arbitrary data $\{ \la_n, \alpha_n\}_{n \ge 1}$ satisfying the hypothesis of Theorem~\ref{stability_thm}, to prove that $\{ \la_n, \alpha_n\}_{n \ge 1}$ are the spectral data of $L$, and in parallel to deduce the estimates \eqref{stab}.

Note that, in general, the problem $L$ can have a finite number of multiple eigenvalues. Then, the spectral data have to be defined in a slightly different way (see \cite{ChitBond}). Furthermore, there is a challenge to study local solvability and stability of the inverse problem taking the possible splitting of multiple eigenvalues into account. The authors plan to investigate these issues in the future by using the ideas of the studies \cite{BK19, Bond20} and by developing the methods of this paper. In Section~\ref{sec:mult}, we consider the special case when multiple eigenvalues are possible but they are not perturbed. Our reconstruction formulas and Theorem~\ref{stability_thm} can be easily generalized to this case, which is important for future research.

The paper is organized as follows. Section~\ref{sec:prelim} contains preliminaries. In Section~\ref{sec:rec}, we derive the reconstruction formulas for the coefficients of the boundary value problem $L$. In Section~\ref{sec:maineq}, we provide the main equation, which has been derived in \cite{ChitBond}, and get some useful estimates for the operator in this equation. In Section~\ref{sec:solvstab}, we prove the main theorem about local solvability and stability of the inverse problem. In Section~\ref{sec:mult}, the special case for multiple eigenvalues is considered.

Throughout the paper, we use the same symbol $C$ for various positive constants independent of $x$, $\la$, $n$, etc.

\section{Preliminaries} \label{sec:prelim}
In this section, we consider some important solutions of equation \eqref{eqv} and the Weyl function of the spectral problem. Moreover, we obtain some its useful properties.

Denote by $\varphi(x, \la)$ and $\psi(x, \la)$ solutions of equation \eqref{eqv} satisfying the initial conditions
$$
\varphi(0, \la) = 1, \quad \varphi^{[1]}(0, \la) = 0, \quad \psi(\pi, \la) = r_1(\la), \quad \psi^{[1]}(\pi, \la) = -r_2(\la).
$$

Clearly, for each fixed $x \in [0,\pi]$, the functions $\varphi(x, \la)$, $\varphi^{[1]}(x, \la)$, $\psi(x, \la)$, and $\psi^{[1]}(x, \la)$ are entire in $\la$. The following proposition provides the useful integral representations of $\varphi(x, \la)$ and $\varphi^{[1]}(x, \la)$.

\begin{prop}[\cite{BondTamkang}] \label{transform_varphi}
The functions $\varphi(x, \la)$ and $\varphi^{[1]}(x, \la)$ can be represented as follows:
\begin{gather*}
    \varphi(x, \la) = \cos\rho x + \int_0^{x}  \mathcal{K}_1(x, t) \cos\rho t \, dt, \\
    \varphi^{[1]}(x, \la) = -\rho\sin\rho x + \rho\int_0^{x} \mathcal{K}_2(x, t) \sin\rho t \, dt + \mathcal{C}(x),
\end{gather*}
where the functions $\mathcal{K}_j(x, t)$, $j = 1, 2$, are square integrable in the region $\{(x, t): 0 < t < x < \pi\}$ and fulfill the conditions $\mathcal K(x, .) \in L_2(0,x)$ for each fixed $x \in (0,\pi]$ and $\sup_x \| \mathcal K(x, .) \|_{L_2(0,x)} < \infty$, the function $\mathcal{C}(x)$ is continuous on $[0, \pi]$. \end{prop}

It can be shown that the eigenvalues $\{ \la_n \}_{n \ge 1}$ of the problem\eqref{eqv}-\eqref{bc} coincide with zeros of the characteristic function
\begin{gather}\label{charfun}
\Delta(\la) = -\psi^{[1]}(0, \la) = r_1(\la)\varphi^{[1]}(\pi, \la) + r_2(\la)\varphi(\pi, \la).
\end{gather}

Note that, for each eigenvalue $\la_n$, we have the unique eigenfunction up to a constant factor. On the one hand, the eigenfunction is proportional to $\varphi(x, \la_n)$, on the other hand, to $\psi(x, \la_n)$, so the following expression holds:
\begin{gather}\label{beta}
\psi(x, \la_n) = \beta_n\varphi(x, \la_n),
\end{gather}
where $\beta_n$ is a nonzero constant.

Let us also introduce the Weyl solution $\Phi(x, \la)$ of equation \eqref{eqv} satisfying the boundary conditions
\begin{gather}\label{WeylsolBC}
\Phi^{[1]}(0,  \la) = 1, \quad r_1(\la)\Phi^{[1]}(\pi, \la) + r_2(\la)\Phi(\pi, \la)=0
\end{gather}
and the Weyl function of the problem $L$:
\begin{gather} \label{defM1}
M(\la) = \Phi(0, \la) = -\dfrac{\psi(0,\la)}{\Delta(\la)}.
\end{gather}

The functions $\Phi(x, \la)$, $\Phi^{[1]}(x, \la)$, and $M(\la)$ are meromorphic in $\la$. Moreover, all their singularities are simple poles which coincide with the eigenvalues $\{ \la_n \}_{n \ge 1}$. The Weyl function is uniquely specified by its poles and residues by the formula (see \cite{ChitBond}):
\begin{gather*}
M(\la) = \sum_{n = 1}^{\infty} \frac{\alpha_n}{\la - \la_n}.
\end{gather*}

Consequently, Inverse Problem~\ref{ip:main3} is equivalent to the recovery of $\sigma(x)$, $r_1(\la)$, and $r_2(\la)$ from the Weyl function $M(\la)$.

Introduce the notations:
$$
\langle y(x), z(x) \rangle = y(x)z^{[1]}(x) - y^{[1]}(x)z(x), \quad r^{\langle1\rangle}(\la) = \dfrac{d r(\la)}{d \la}.
$$

It can be shown by direct calculations that, if $y$ and $z$ are solutions of equation \eqref{eqsi} with spectral parameters $\la$ and $\mu$, respectively, then
\begin{equation} \label{wron}
\frac{d}{dx} \langle y(x),  z(x) \rangle = (\la - \mu) y(x) z(x).
\end{equation}

Below we need the following auxiliary lemma.

\begin{lem}\label{lemforz}
	The following relation holds:
	$$
	- r_1^{\langle1\rangle}(\la_{n})\varphi^{[1]}(\pi, \la_{n})-r_2^{\langle1\rangle}(\la_{n})\varphi(\pi,\la_{n}) = \frac{r_1(\la_{n})}{\varphi(\pi,\la_{n})}\Bigg(\int_0^\pi {\varphi^2(x,\la_{n})dx} + \frac{1}{\alpha_{n}}\Bigg), \quad n \ge 1.
	$$
\end{lem}

\begin{proof}
	The relation \eqref{wron} for $\psi(x, \la)$ and $\varphi(x, \la_n)$ implies
	$$
	\frac{d}{dx}\langle\psi(x, \la), \varphi(x, \la_{n})\rangle = (\la - \la_{n})\psi(x, \la)\varphi(x, \la_{n}).
	$$
	Integrating the both sides, we get
	\begin{align}\notag
	(\la-\la_{n})\int\limits_0^\pi{\psi(x, \la)\varphi(x, \la_{n})dx} = \langle\psi(x, \la), \varphi(x, \la_{n})\rangle\bigg|_0^\pi = \\ \notag
	-\psi^{[1]}(0, \la) - r_1(\la)\varphi^{[1]}(\pi, \la_{n}) - r_2(\la)\varphi(\pi, \la_{n}).
	\end{align}
	Dividing the both sides by $(\la - \la_{n})$ and using \eqref{charfun}, we obtain
	\begin{gather}\label{intlem}
	\int\limits_0^\pi{\psi(x, \la)\varphi(x, \la_{n})dx} = \frac{ \Delta(\la) - r_1(\la)\varphi^{[1]}(\pi, \la_{n}) - r_2(\la)\varphi(\pi, \la_{n})}{\la - \la_{n}}.
	\end{gather}
	Passing to the limits as $\la$ tends to $\la_{n}$ and using \eqref{beta}, we get
	\begin{gather}\label{left}
	\lim_{\lambda \to \lambda_n} \int\limits_0^\pi{\psi(x, \la)\varphi(x, \la_{n})dx = \beta_n\int\limits_0^\pi{\varphi^2(x, \la_{n})dx} = \frac{r_1(\la_{n})}{\varphi(\pi, \la_n)}\int\limits_0^\pi{\varphi^2(x, \la_{n})dx}}.
	\end{gather}
	Due to L'H\^opital's rule, the right-hand side of \eqref{intlem} turns into
	\begin{gather}\notag
	\lim_{\lambda \to \lambda_n}\frac{\Delta(\la) - r_1(\la)\varphi^{[1]}(\pi, \la_{n}) - r_2(\la)\varphi(\pi, \la_{n})}{\la - \la_{n}}=\Delta^{\langle1\rangle}(\la_{n}) - r_1^{\langle1\rangle}(\la_{n})\varphi^{[1]}(\pi,\la_{n})-r_2^{\langle1\rangle}(\la_{n}) \varphi(\pi, \la_{n}).
	\end{gather}
	The value $\Delta^{\langle1\rangle}(\la_{n})$ can be found from \eqref{defM1} and \eqref{beta}: 
	$$
	\mathop{\mathrm{Res}}\limits_{\la = \la_{n}} M(\la) = -\frac{\psi(0, \la_{n})}{\Delta^{\langle1\rangle}(\la_{n})} = -\frac{\beta_n}{\Delta^{\langle1\rangle}(\la_{n})} = \alpha_{n}.
	$$
	Hence
	$$
	\Delta^{\langle1\rangle}(\la_{n}) = -\frac{\beta_n}{\alpha_{n}} = -\frac{r_1(\la_{n})}{\alpha_{n}\varphi(\pi, \la_{n})}.
	$$
	So, we get
	\begin{multline} \label{right}
	\lim_{\lambda \to \lambda_n}\frac{ \Delta(\la) - r_1(\la)\varphi^{[1]}(\pi, \la_{n}) - r_2(\la)\varphi(\pi, \la_{n})}{\la - \la_{n}} \\ = -\frac{r_1(\la_{n})}{\alpha_{n}\varphi(\pi, \la_{n})}
	-  r_1^{\langle1\rangle}(\la_{n})\varphi^{[1]}(\pi, \la_{n})- r_2^{\langle1\rangle}(\la_{n})\varphi(\pi, \la_{n}).
	\end{multline}
	Substituting \eqref{left} and \eqref{right} into \eqref{intlem}, we arrive at the assertion of the lemma.
\end{proof}

\section{Reconstruction formulas} \label{sec:rec}

In this section, we derive reconstruction formulas for the solution of Inverse Problem~\ref{ip:main3}. For this purpose, we consider two boundary value problems $L = L(\sigma, r_1, r_2)$ and $\tilde L = L(\tilde\sigma, \tilde r_1, \tilde r_2)$. Assume that $\sigma, \tilde \sigma \in L_2(0,\pi)$, $(r_1, r_2), (\tilde r_1, \tilde r_2) \in \mathcal R$, $M_1 = \tilde M_1$, and the eigenvalues of $L$ and $\tilde L$ are simple. Note that the quasi-derivatives corresponding to $L$ and $\tilde L$ are different: $y^{[1]} = y' - \sigma y$ and $\tilde y^{[1]} = y' - \tilde \sigma y$, respectively. We will use $\tilde L$ as the so-called \textit{model problem} and recover the coefficients of the problem $L$. Note that, for the zero model problem : $\tilde \sigma(x) \equiv 0$, $\tilde r_1(\la) = \la^{M_1}$, $\tilde r_2(\la) = 0$ the reconstruction formulas have been obtained in \cite{ChitBond}. However, for the proof of Theorem~\ref{stability_thm}, we need reconstruction formulas for the case of an arbitrary model problem $\tilde L$. Thus, the reconstruction formulas and their derivation become more technically complicated than the ones from \cite{ChitBond}.

Introduce the notations
\begin{equation} \notag
\def\arraystretch{1.5}
\left.
\begin{array}{c}
\la_{n0} = \la_n, \quad \la_{n1} = \tilde \la_n, \quad \rho_{n0} = \rho_n, \quad \rho_{n1} = \tilde \rho_n, \quad \alpha_{n0} = \alpha_n, \quad \alpha_{n1} = \tilde \alpha_n, \\
\varphi_{ni}(x) = \varphi(x, \la_{ni}), \quad \tilde \varphi_{ni}(x) = \tilde \varphi(x, \la_{ni}).
\end{array} \quad \right\}
\end{equation}

Let us get reconstruction formulas for $\sigma(x)$, $r_1(\la)$, $r_2(\la)$, assuming that $\tilde \sigma(x)$, $\tilde r_1(\la)$, $\tilde r_2(\la)$ are known.

\begin{thm} \label{mainthm}
The solution of Inverse Problem~\ref{ip:main3} can be found by the reconstruction formulas
	\begin{align} \label{sigma_series}
	\sigma(x) = \tilde \sigma(x) - 2 \sum_{k=1}^\infty \sum_{j=0}^1 {(-1)^j\alpha_{kj}\Bigg(\tilde \varphi_{kj}(x)\varphi_{kj}(x) - \frac{1}{2}\Bigg)},
	\end{align}
	\begin{align} \label{r1_series}
	r_1(\la) = \prod_{k = 1}^{\infty} \frac{\la - \la_{k0}}{\la - \la_{k1}}\Bigg(\tilde r_1(\la) - \sum_{k=1}^\infty \sum_{j=0}^1 \frac{(-1)^j\alpha_{kj}\varphi_{kj}(\pi)(\tilde r_1(\la)\tilde\varphi^{[1]}_{kj}(\pi) + \tilde r_2(\la)\tilde\varphi_{kj}(\pi))}{\la - \la_{kj}} \Bigg),
	\end{align}
	\begin{align}\notag 
	r_2(\la) &= \prod_{k = 1}^{\infty} \frac{\la - \la_{k0}}{\la - \la_{k1}}\Bigg(\tilde r_2(\la) - \tilde r_1(\la) \sum_{k=1}^\infty \sum_{j=0}^1 {(-1)^j\alpha_{kj}\Bigg(\tilde \varphi_{kj}(\pi)\varphi_{kj}(\pi) - 1\Bigg)} \\ \label{r2_series}
	&+ \sum_{k = 1}^{\infty} \sum_{j=0}^1 \frac{(-1)^j\alpha_{kj}\varphi^{[1]}_{kj}(\pi)(\tilde r_1(\la)\tilde\varphi^{[1]}_{kj}(\pi) + \tilde r_2(\la)\tilde\varphi_{kj}(\pi))}{\la - \la_{kj}} \Bigg),
	\end{align}
	where the series in \eqref{sigma_series} converges in $L_2(0,\pi)$ and the series and the infinite products in \eqref{r1_series} and \eqref{r2_series} converge absolutely and uniformly on compact sets in $\mathbb C \setminus \{ \la_{ni}\}_{n \ge 1, \, i = 0, 1}$.
\end{thm}

In this section, we do not provide a rigorous proof of Theorem~\ref{main_equation_thm}, since a more general fact is actually proved in Section~\ref{sec:solvstab}.
Specifically, for arbitrary data $\{ \la_n, \alpha_n \}_{n \ge 1}$ satisfying the hypothesis of Theorem~\ref{stability_thm} (not necessarily being the spectral data of some problem $L$), we prove the convergence of the series in~\eqref{sigma_series}--\eqref{r2_series}. Next, it is shown that the right-hand sides of \eqref{r1_series} and \eqref{r2_series} are polynomials of suitable degrees and that the problem $L$ with the coefficients defined by the right-hand sides of \eqref{sigma_series}--\eqref{r2_series} has the spectral data $\{ \la_n, \alpha_n \}_{n \ge 1}$. 
Here, we only describe the idea of deriving the reconstruction formulas.

By using the contour integration in the $\la$-plane, we have obtained the relations between the solutions of $L$ and $\tilde L$ in \cite{ChitBond}.

\begin{prop}[\cite{ChitBond}]
	The following representations hold:
	\begin{gather} \label {series_varphi}
	\varphi(x, \la) = \tilde \varphi(x, \la) - \sum_{k=1}^\infty\bigg( \alpha_{k0}\tilde D(x, \la, \la_{k0})\varphi_{k0}(x) - \alpha_{k1}\tilde D(x, \la, \la_{k1})\varphi_{k1}(x) \bigg),\\ \label {series_phi}
	\Phi(x, \la) = \tilde \Phi(x, \la) - \sum_{k=1}^\infty\bigg( \alpha_{k0}\tilde E(x, \la, \la_{k0})\varphi_{k0}(x) - \alpha_{k1}\tilde E(x, \la, \la_{k1})\varphi_{k1}(x) \bigg),
	\end{gather}
	where
	\begin{equation} \label{defDE}
	\tilde D(x, \la, \mu) = \frac {\langle \tilde \varphi(x, \la), \tilde \varphi(x, \mu) \rangle}{\la-\mu}, \quad \tilde E(x, \la, \mu) = \frac {\langle \tilde \Phi(x, \la), \tilde \varphi(x, \mu) \rangle}{\la-\mu}.
	\end{equation}
\end{prop}

The reconstruction formulas of Theorem~\ref{mainthm} are obtained from the relations \eqref{series_varphi} and \eqref{series_phi} by formal calculations. Thus, we do not prove here convergence of series and existence of derivatives. These issues will be rigorously proved in Section~\ref{sec:solvstab}.

Substituting \eqref{series_varphi} into the \eqref{eqv} and using the relation $\sigma(x) = \int_0^x {q(x)dx} + \sigma(0)$, we obtain \eqref{sigma_series}. Proceed to the derivation of \eqref{r1_series} and \eqref{r2_series}.
From \eqref{WeylsolBC} we get:
\begin{gather}\label{weylsolfrac}
\frac{r_2(\la)}{r_1(\la)} = -\frac{\Phi^{[1]}(\pi, \la)}{\Phi(\pi, \la)},\\ \label{tildeweylsolfrac}
\frac{\tilde r_2(\la)}{\tilde r_1(\la)} = -\frac{\tilde\Phi^{[1]}(\pi, \la)}{\tilde\Phi(\pi, \la)}.
\end{gather}

Using \eqref{wron} and \eqref{defDE}, one can show that
$$
\tilde E'(x, \la, \mu) = \tilde \Phi(x, \la) \tilde \varphi(x, \mu).
$$

Therefore, differentiating \eqref{series_phi}, we obtain
\begin{align} \nonumber
\Phi^{[1]}(x, \la)& = \Phi'(x, \la) - \sigma(x)\Phi(x, \la) \\ \nonumber
&= \tilde \Phi^{[1]}(x, \la) + \Bigg(\tilde \sigma(x) - \sigma(x) - \sum\limits_{k=1}^\infty\sum\limits_{j=1}^1 (-1)^j\alpha_{kj}\tilde\varphi_{kj}(x)\varphi_{kj}(x)\Bigg)\tilde \Phi(x, \la)\\ \label{Phiquas}
& - \sum\limits_{k=1}^\infty\sum\limits_{j=1}^1 (-1)^j\alpha_{kj}\tilde E(x, \la, \la_{kj})\varphi^{[1]}_{kj}(x).
\end{align}

Substituting \eqref{series_phi} and \eqref{Phiquas} into \eqref{weylsolfrac} and using \eqref{sigma_series}, we get:
\begin{align}\notag
\frac{r_2(\la)}{r_1(\la)} =& -\Bigg(\tilde\Phi^{[1]}(\pi, \la)+\tilde\Phi(\pi, \la)\sum\limits_{k=1}^{\infty}\sum\limits_{j=0}^{1}(-1)^j\alpha_{kj}(\tilde\varphi_{kj}(\pi)\varphi_{kj}(\pi)-1)\\ 
&-\notag\sum\limits_{k=1}^{\infty}\sum\limits_{j=0}^{1}(-1)^j\alpha_{kj}\tilde E(\pi, \la, \la_{kj})\varphi_{kj}^{[1]}(\pi)\Bigg)\Bigg(\tilde \Phi(\pi, \la) - \sum\limits_{k=1}^\infty\sum\limits_{j=0}^1(-1)^j \alpha_{kj}\tilde E(\pi, \la, \la_{kj})\varphi_{kj}(\pi)\Bigg)^{-1}.
\end{align}

Dividing numerator and denominator by $\tilde\Phi(\pi, \la)$ and taking \eqref{tildeweylsolfrac} into account, we obtain
\begin{gather*}
\frac{r_2(\la)}{r_1(\la)} = \frac{C(\la)}{Z(\la)},
\end{gather*}
where
\begin{align} \notag
C(\la) &= \tilde r_2(\la) - \tilde r_1(\la)\sum\limits_{k=1}^{\infty}\sum\limits_{j=0}^{1}(-1)^j\alpha_{kj}(\tilde\varphi_{kj}(\pi)\varphi_{kj}(\pi)-1)\\ \notag
&+\sum\limits_{k=1}^\infty\sum\limits_{j=0}^1{\frac{(-1)^j\alpha_{kj}\varphi^{[1]}_{kj}(\pi)(\tilde r_1(\la) \tilde\varphi^{[1]}_{kj}(\pi) + \tilde r_2(\la) \tilde\varphi_{kj}(\pi))}{\la-\la_{kj}}},
\end{align}
\begin{align} \label{denominator}
Z(\la) = \tilde r_1(\la) - \sum\limits_{k=1}^\infty\sum\limits_{j=0}^1{\frac{(-1)^j\alpha_{kj}\varphi_{kj}(\pi)(\tilde r_1(\la) \tilde\varphi^{[1]}_{kj}(\pi) + \tilde r_2(\la) \tilde\varphi_{kj}(\pi))}{\la-\la_{kj}}}.
\end{align}

In order to obtain polynomials from $C(\la)$ and $Z(\la)$, we need the following technical lemma.

\begin{lem}\label{lem:cz}
$C(\la_{n1})=Z(\la_{n1})=0$, $n \ge 1$.
\end{lem}

\begin{proof}
From \eqref{denominator}, we get
\begin{align} \notag
Z(\la_{n1}) = \tilde r_1(\la_{n1}) &- \sum\limits_{k=1}^\infty\Bigg(\frac{\alpha_{k0}\varphi_{k0}(\pi)(\tilde r_1(\la_{n1}) \tilde\varphi^{[1]}_{k0}(\pi) + \tilde r_2(\la_{n1}) \tilde\varphi_{k0}(\pi))}{\la_{n1}-\la_{k0}} \\ \notag &-\frac{\alpha_{k1}\varphi_{k1}(\pi)(\tilde r_1(\la_{n1}) \tilde\varphi^{[1]}_{k1}(\pi) + \tilde r_2(\la_{n1}) \tilde\varphi_{k1}(\pi))}{\la_{n1}-\la_{k1}}(1-\delta_{nk})\Bigg) \\ \notag & + \alpha_{n1}\varphi_{n1}(\pi)(\tilde r_1^{\langle1\rangle}(\la_{n1})\tilde\varphi^{[1]}_{n1}(\pi)+\tilde r_2^{\langle1\rangle}(\la_{n1}) \tilde\varphi_{n1}(\pi)),
\end{align}
where $\delta_{nk}$ is the Kronecker delta.

The relations \eqref{wron} and \eqref{defDE} imply
$\tilde D(x, \la, \mu) = \int\limits_0^x \tilde \varphi(t, \la)\tilde \varphi(t, \mu) \,dt$. Using this integral form, we get from \eqref{series_varphi} that
\begin{align} \notag
\varphi_{n1}(\pi) = \tilde \varphi_{n1}(\pi) &- \sum_{k=1}^\infty\bigg( \alpha_{k0}\tilde D(\pi, \la_{n1}, \la_{k0})\varphi_{k0}(\pi) - \alpha_{k1}\tilde D(\pi, \la_{n1}, \la_{k1})\varphi_{k1}(\pi)(1-\delta_{nk}) \bigg)\\ \notag
 &+ \alpha_{n1}\varphi_{n1}(\pi)\int_0^\pi {\tilde\varphi^2_{n1}(t)dt}.
\end{align}

Multiplying the latter equality by $\dfrac{\tilde r_1(\la_{n1})}{\tilde\varphi_{n1}(\pi)}$, we obtain
$$
Z(\la_{n1}) - \alpha_{n1}\varphi_{n1}(\pi)(\tilde r_1^{\langle1\rangle}(\la_{n1})\tilde\varphi^{[1]}_{n1}(\pi)+\tilde r_2^{\langle1\rangle}(\la_{n1}) \tilde\varphi_{n1}(\pi)) = \frac{\varphi_{n1}(\pi)\tilde r_1(\la_{n1})}{\tilde\varphi_{n1}(\pi)}\Bigg(\alpha_{n1}\int_0^\pi {\tilde\varphi^2_{n1}(t)dt} + 1\Bigg).
$$

Due to Lemma~\ref{lemforz}, we get that $Z(\la_{n1}) = 0$. In the same way we can show that $C(\la_{n1}) = 0$.
\end{proof}

According to Lemma~\ref{lem:cz}, $C(\la)$ and $Z(\la)$ multiplied by $\prod\limits_{k = 1}^{\infty} \dfrac{\la - \la_{k0}}{\la - \la_{k1}}$ are entire functions. Moreover, it will be shown in the proof of Theorem~\ref{stability_thm} that they are a pair of polynomials belonging to $\mathcal R$. This implies the reconstruction formulas \eqref{r1_series} and \eqref{r2_series}.

\section{Main equation} \label{sec:maineq}

In this section, we get the main equation of the problem $L$ and consider its properties, which are used in the next section to prove useful estimates.

As in the previous section, we continue considering the two problems $L = L(\sigma, r_1, r_2)$ and $\tilde L = L(\tilde \sigma, \tilde r_1,\tilde r_2)$.
Denote
\begin{gather} \label{defden}
\delta_n := |\rho_n - \tilde \rho_n| + |\alpha_n - \tilde \alpha_n|, \quad
\chi_{n} := \left\{\begin{array}{ll}
\delta_n^{-1}, \quad & \delta_n \neq 0,\\
0, \quad & \delta_n = 0,
\end{array}\right. \quad n \ge 1.
\end{gather}

According to the spectral data asymptotics \eqref{la_asymp} and \eqref{alpha_asymp}, we have $\{ \delta_n \} \in l_2$.

For $\la = \la_{ni}$, the relation \eqref{series_varphi} implies
\begin{equation} \label{main_varphi}
\tilde \varphi_{ni}(x) = \varphi_{ni}(x) + \sum_{k = 1}^{\infty} (\tilde Q_{ni; k0}(x) \varphi_{k0}(x) - \tilde Q_{ni;k1} \varphi_{k1}(x)), \quad n \ge 1, \, i = 0, 1,
\end{equation}
where 
\begin{equation} \label{defQ}
\tilde Q_{ni; kj}(x) = \alpha_{kj}\tilde D(x, \la_{ni}, \la_{kj}).
\end{equation}
 
Define the block vector $\varphi(x) = (\varphi_n(x))_{n \ge 1}$ and the block matrix $\tilde Q(x) = (\tilde Q_{n, k}(x))_{n, k \ge 1}$ as follows:
$$
\varphi_n(x) =
\begin{pmatrix}
\varphi_{n,0}(x)\\
\varphi_{n,1}(x)
\end{pmatrix}, \qquad
\tilde Q_{n, k}(x) = 
\begin{pmatrix}
\tilde Q_{n0; k0}(x) & -\tilde Q_{n0; k1}(x)\\
\tilde Q_{n1; k0}(x) & -\tilde Q_{n1; k1}(x)
\end{pmatrix}.
$$
The block vector $\tilde \varphi(x)$ is defined analogously to $\varphi(x)$ replacing $\varphi_{n,i}(x)$ by $\tilde \varphi_{n,i}(x)$. Then,
the system \eqref{main_varphi} can be rewritten as
\begin{gather}\label{varphi_q_eqv}
\tilde\varphi(x) = (E+\tilde Q(x))\varphi(x),
\end{gather}
where $E$ is the infinite identity matrix.

Note that the series in \eqref{main_varphi} converges only ``with brackets''. In order to achieve the absolute convergence, we use the transform induced by the block row vector $T = (T_{n})_{n \ge1}$:
$$
T_{n} = 
\begin{pmatrix}
\chi_n & -\chi_n\\
0 & 1
\end{pmatrix}.
$$
Put 
\begin{equation} \label{defH}
\psi(x) = (\psi_n(x))_{n \ge 1} := T \varphi(x), \qquad \tilde H(x) = (\tilde H_{n,k}(x))_{n,k\ge 1} := T \tilde Q(x) T^{-1}.
\end{equation}
In the element-wise form, this means
\begin{gather} \label{psimain}
\psi_n(x) = 
	\begin{pmatrix}
	\psi_{n0}(x)\\
	\psi_{n1}(x)
	\end{pmatrix} =
	\begin{pmatrix}
	\chi_{n} & -\chi_{n}\\
	0 & 1
	\end{pmatrix}
	\begin{pmatrix}
	\varphi_{n,0}(x)\\
	\varphi_{n,1}(x)
	\end{pmatrix} = T_n\varphi_n(x), \\ \label{Hmatr}
	\begin{pmatrix}
	\tilde H_{n0, k0}(x) & \tilde H_{n0, k1}(x)\\
	\tilde H_{n1, k0}(x) & \tilde H_{n1, k1}(x)
	\end{pmatrix} =
	\begin{pmatrix}
	\chi_{n} & - \chi_{n}\\
	0 & 1
	\end{pmatrix}
	\begin{pmatrix}
	\tilde Q_{n0; k0}(x) & -\tilde Q_{n0; k1}(x)\\
	\tilde Q_{n1; k0}(x) & -\tilde Q_{n1; k1}(x)
	\end{pmatrix}
	\begin{pmatrix}
	\delta_{k} & 1\\
	0 & 1
	\end{pmatrix}.
\end{gather}
The block vector $\tilde \psi(x) := T \tilde \varphi(x)$ is obtained analogously.

Due to \eqref{psimain} and \eqref{Hmatr}, the relation \eqref{varphi_q_eqv} is reduced to the form $\tilde \psi(x) = (E + \tilde H(x)) \psi(x)$, which can be considered as a linear equation in a Banach space.
Indeed, let $V$ be the set of indices $v = (n, i)$, where $n \ge 1$, $i=0, 1$, and
let $m$ be the Banach space of bounded infinite sequences $a = (a_v)_{v \in V}$ with the norm $\| a \|_m = \sup_{v} |a_v|$.

\begin{thm}[\cite{ChitBond}]\label{main_equation_thm}
For each fixed $x \in [0,\pi]$, the vectors $\psi(x)$ and $\tilde \psi(x)$ belong to $m$ and $\tilde H(x)$ is a bounded linear operator in $m$.
Furthermore, the vector $\psi(x)$ satisfies the equation
\begin{gather}\label{main_equation}
\tilde \psi(x) = (E + \tilde H(x)) \psi(x)
\end{gather}
in the Banach space $m$. 
\end{thm}

Equation \eqref{main_equation} is called \textit{the main equation} of Inverse Problem~\ref{ip:main3}. Note that the vector function $\tilde \psi(x)$ and the operator $\tilde H(x)$ can be constructed by using only the problem $\tilde L$ and the spectral data $\{ \la_n, \alpha_n\}_{n \ge 1}$ of $L$, while $\psi(x)$ is related to the coefficients of $L$. Therefore, one can solve the main equation \eqref{main_equation} to find $\psi(x)$ and then reconstruct $\sigma(x)$, $r_1(\la)$, and $r_2(\la)$ by the formulas \eqref{sigma_series}--\eqref{r2_series}. In in \cite{ChitBond}, this algorithm has been considered in detail  for $\tilde L = L(0, \la^{M_1}, 0)$. 

The main equation \eqref{main_equation} is essential for the proof of Theorem~\ref{stability_thm}. In addition, we need to study some properties of the operator $T^{-1} \tilde H(x) = \tilde Q(x) T^{-1}$.
For this purpose, we need a number of technical lemmas.

\begin{prop}[\cite{He01}] \label{rz_cos}
Let $\{\rho_k\}_{k=1}^\infty$ be an arbitrary sequence of distinct elements: $\rho_k \ne \rho_s$ if $k \neq s$, such that $\rho_k = k +\varkappa_k$, $\{\varkappa_k\} \in l_2$. Then the sequence $\{\cos(\rho_kx)\}_{k=0}^\infty$ is a Riesz basis in $L_2(0, \pi)$. Moreover, if $\{f_k\}_{k=0}^\infty$is any sequence of $l_2$, then
$$
F(x):=\sum\limits_{k=0}^\infty f_k\cos(\rho_kx) \in L_2(0, \pi), \quad \|F\|_{L_2}\le C\|\{f_k\}\|_{l_2};
$$
and if $F(x)$ is any function of $L_2(0, \pi)$, then $\{f_k\}_{k=0}^\infty \in l_2$, where
$$
f_k:=\int\limits_0^\pi F(x)\cos(\rho_kx) \,dx, \quad \|\{f_k\}\|_{l_2}\le C\|F\|_{L_2}.
$$
\end{prop}

Propositions~\ref{transform_varphi} and~\ref{rz_cos} together imply the following lemma.

\begin{lem} \label{rz_varphi}
Let $\{ \rho_k \}_{k = 1}^{\infty}$ be any sequence of distinct complex numbers satisfying the asymptotics \eqref{la_asymp}.	
If $\{f_k\}_{k=1}^\infty$ is any sequence of $l_2$, then
$$
F(x):=\sum\limits_{k=1}^\infty f_k \tilde\varphi(x, \rho_k) \in L_2(0, \pi), \quad \|F\|_{L_2}\le C\|\{f_k\}\|_{l_2};
$$
and if $F(x)$ is any function of $L_2(0, \pi)$, then $\{f_k\}_{k=1}^\infty \in l_2$, where
\begin{gather}\label{fk}
f_k:=\int\limits_0^\pi F(x)\tilde\varphi(x,\rho_k)dx, \quad \|\{f_k\}\|_{l_2}\le C\|F\|_{L_2}.
\end{gather}
\end{lem}

\begin{proof}
From Proposition~\ref{transform_varphi} we get:
$$
F(x)=\sum\limits_{k=1}^\infty f_k\cos(\rho_{k}x) + \sum\limits_{k=1}^\infty f_k\int\limits_0^x\mathcal K_1(x, t)\cos(\rho_{k}t) \,dt =: F_1(x) + F_2(x).
$$

The system $\{\cos(\rho_kx)\}_{k=1}^\infty$ is not minimal, so it is not a Riesz basis. But we can take the subsystem $\{\cos(\rho_kx)\}_{k > M_1}$, which is a Riesz basis, so due to Proposition~\ref{rz_cos}, we get that 
$$
\sum\limits_{k=M_1+1}^\infty f_k\cos(\rho_{k}x) \in L_2(0, \pi), \quad \bigg\|\sum\limits_{k=M_1+1}^\infty f_k\cos(\rho_{k}x)\bigg\|_{L_2}\le C\|\{f_k\}\|_{l_2}.
$$
For the finite remainder the same estimate holds, so $F_1 \in L_2(0, \pi)$ and $\|F_1\|_{L_2}\le C\|\{f_k\}\|_{l_2}$.
As $\{\cos(\rho_kx)\}_{k > M_1}$ is a Riesz basis, then $\int\limits_0^x\mathcal K_1(x, t)\cos(\rho_{k}t)dt$ gives the Fourier coefficients for $\mathcal K_1(x, t)$ for each fixed $x \in [0, \pi]$. Using the properties of $\mathcal K_1(x, t)$ from Proposition~\ref{transform_varphi}, we conclude that
$$
\Bigg\|\Bigg\{\int\limits_0^x\mathcal K_1(x, t)\cos(\rho_{k}t) \,dt\Bigg\}\Bigg\|_{l_2} \le C\|\mathcal K_1(x, .)\|_{L_2(0, x)} < \infty.
$$

Then, we get:
\begin{gather*}
\sum\limits_{k=M_1+1}^\infty f_k\int\limits_0^x\mathcal K_1(x, t)\cos(\rho_{k}t) \,dt \in L_2(0, \pi), \\ \bigg\|\sum\limits_{k=M_1+1}^\infty f_k\int\limits_0^x\mathcal K_1(x, t)\cos(\rho_{k}t) \,dt\bigg\|_{L_2}\le C\|\{f_k\}\|_{l_2}.
\end{gather*}
For the finite remainder the same estimate holds, so $F_2 \in L_2(0, \pi)$ and $\|F_2\|_{L_2}\le C\|\{f_k\}\|_{l_2}$. Thus, the first part of the lemma is proved.

Proceed to the second part.
From Proposition~\ref{transform_varphi}, we get:
$$
f_k=\int\limits_0^\pi F(x)\cos(\rho_{k}x) \,dx + \int\limits_0^\pi F(x)\int\limits_0^x\mathcal K_1(x, t)\cos(\rho_{k}t) \,dt\,dx =: f_{k1} + f_{k2}.
$$

The sequences $\{f_{k1}\}_{k>M_1}$ consists of the Fourier coefficients for function $F(x)$ in the basis $\{\cos(\rho_kx)\}_{k > M_1}$, so using Proposition~\ref{rz_cos} we get that $\|\{f_{k1}\}\|_{l_2}\le C\|F\|_{L_2}$. Using the same ideas, we can get that $\|\{f_{k2}\}\|_{l_2}\le C\|F\|_{L_2}$.
This yields the claim.
\end{proof}

\begin{lem}\label{lem_qt}
For each fixed $x \in [0, \pi]$, the operator $\tilde Q(x)T^{-1}$ maps $m$ into $l_2$ and is uniformly bounded by $x \in [0, \pi]$.
\end{lem}
\begin{proof}
Let $x \in [0, \pi]$ be fixed. Let $f$ be an element from Banach space $m$. So,
$$
T^{-1}_nf_n = 
\begin{pmatrix}
	\delta_n f_{n0} + f_{n1}\\
	f_{n1}
\end{pmatrix}
$$
and
$$
\tilde Q_{n, k}(x)T^{-1}_kf_k = 
\begin{pmatrix}
	\tilde Q_{n0;k0}(x)(\delta_k f_{k0} + f_{k1}) + \tilde Q_{n0;k1}(x)f_{k1}\\
	\tilde Q_{n1;k0}(x)(\delta_k f_{k0} + f_{k1}) + \tilde Q_{n1;k1}(x)f_{k1}
\end{pmatrix}.
$$

Recall that $\tilde D(x, \la, \mu)$ can be represented in the integral form:
$$
\tilde D(x, \la, \mu) = \int\limits_0^x{\tilde\varphi(t, \la)\tilde\varphi(t, \mu)dt}.
$$

Using this formula and \eqref{defQ}, we represent the entries of
$\tilde Q(x)T^{-1}f$ as follows:
\begin{align} \notag
\bigg(\tilde Q(x)T^{-1}f\bigg)_{ni} &= \sum\limits_{k}^{}\Bigg(\delta_k f_{k0}\alpha_{k0}\int\limits_0^x{\tilde\varphi_{ni}(t)\tilde\varphi_{k0}(t)dt} + f_{k1}\alpha_{k0}\int\limits_0^x{\tilde\varphi_{ni}(t)\tilde\varphi_{k0}(t)dt} \\ \notag
&- f_{k1}\alpha_{k1}\int\limits_0^x{\tilde\varphi_{ni}(t)\tilde\varphi_{k1}(t)dt}\Bigg) \\ \notag
&= 
\sum\limits_{k}^{}\int\limits_0^x\tilde\varphi_{ni}(t)\Bigg( \delta_k f_{k0}\alpha_{k0}\tilde\varphi_{k0}(t) + f_{k1}\bigg( \alpha_{k0}\tilde\varphi_{k0}(t) - \alpha_{k1}\tilde\varphi_{k1}(t)  \bigg) \Bigg)dt.
\end{align}

Put
$$
F(t) = \sum\limits_{k}^{}\Bigg( \delta_k f_{k0}\alpha_{k0}\tilde\varphi_{k0}(t) + f_{k1}\bigg( \alpha_{k0}\tilde\varphi_{k0}(t) - \alpha_{k1}\tilde\varphi_{k1}(t)  \bigg) \Bigg).
$$
Using \eqref{la_asymp}, \eqref{alpha_asymp}, the estimate $|f_{kj}| \le \|f\|_m$, and Proposition~\ref{transform_varphi}, we conclude that 
\begin{equation} \label{estF}
F(t) \in L_2(0, \pi), \qquad \| F \|_{L_2(0,\pi)} \le C \| f \|_m. 
\end{equation}
So, 
\begin{equation} \label{QF}
\bigg(\tilde Q(x)T^{-1}f\bigg)_{ni} = \int\limits_0^x {F(t)\tilde\varphi_{ni}(t)}dt.
\end{equation}
According to Proposition~\ref{rz_varphi}, the obtained expression gives the Fourier coefficients of the functions $F_{[0, x]}(t)$, where
\begin{gather}\notag
F_{[0, x]}(t) =
\left\{\begin{array}{ll}
        F(t), \quad & t \in (0, x), \\
        0, \quad & t \in (x, \pi). \\
    \end{array}\right.
\end{gather}
Hence, $\tilde Q(x)T^{-1}f \in l_2$ and
$$
\left\| \tilde Q(x)T^{-1}f \right\|_{l_2} \le C \|F_{[0, x]}\|_{L_2(0, \pi)} \le C\|F\|_{L_2(0, \pi)} \le C\|f\|_m.
$$
Clearly, the constant $C$ does not depend on $x \in [0,\pi]$. This yields the claim.
\end{proof}

\begin{lem} 
Let $f \in m$ and $q_{ni}(x) = \bigg(\tilde Q(x)T^{-1}f\bigg)_{ni}$. Then $\{(q_{n0} - q_{n1})(x)\}_{n=1}^\infty \in l_1$ and $\|\{(q_{n0} - q_{n1})(x)\}\|_{l_1} \le C \|f\|_m$.
\end{lem}

\begin{proof}
Denote $\zeta_n := \rho_{n0} - \rho_{n1}$ and consider the difference
\begin{align}\label{difcos}
\cos(\rho_{n0}t) - \cos(\rho_{n1}t) & = - 2 \sin (\rho_{n1} t) \sin \frac{\zeta_n t}{2} = - \zeta_n t \sin (\rho_{n1}t) + O(\zeta_n^2).
\end{align}
In view of \eqref{defden}, we have $|\zeta_n| \le \delta_n$.
Therefore, using Lemma~\ref{lem_qt}, Proposition~\ref{transform_varphi} and \eqref{difcos}, we get
\begin{align}\notag
|(q_{n0} - q_{n1})(x)| & = \bigg|\int\limits_0^x {F(t)(\tilde\varphi_{n0}(t)-\tilde\varphi_{n1}(t))dt}\bigg|\\ \notag
& = \bigg|\zeta_n\int\limits_0^x {F(t)\bigg(-t\sin(\rho_{n1}t)-\int\limits_0^t{ \mathcal{K}_1(t, u)u\sin(\rho_{n1}u) \, du}\bigg)dt}\bigg| + O(\zeta_n^2)\bigg|\int\limits_0^x {F(t) dt} \bigg| \\ \notag
& = |\delta_nk_n(x)| + |s_n(x)|,
\end{align}
where
$$
|s_n(x)| \le C\delta_n^2\|F\|_{L_2}, \quad \bigg(\sum\limits_{n=1}^\infty|k_n(x)|^2\bigg)^{\frac{1}{2}} \le C\|F\|_{L_2}
$$
uniformly by $x \in [0,\pi]$.

Then
$$
\sum\limits_{n=1}^\infty|q_{n0} - q_{n1}(x)| \le \sum\limits_{n=1}^\infty|\delta_nk_n(x)| + \sum\limits_{n=1}^\infty|s_n(x)| \le C\|F\|_{L_2}.
$$

This together with the estimate \eqref{estF} yield the claim.
\end{proof}

\section{Local solvability and stability} \label{sec:solvstab}

In this section, we prove Theorem~\ref{stability_thm} on local solvability and stability of Inverse Problem~\ref{ip:main3}. The proof consists of several lemmas.
The method of the proof is based on the main equation \eqref{main_equation} and on the reconstruction formulas \eqref{sigma_series}--\eqref{r2_series}. First, we construct the main equation by the data satisfying the hypothesis of Theorem~\ref{stability_thm} and prove the unique solvability of the main equation. Second, we investigate the properties of the solution $\psi(x)$. Third, by using the  entries of $\psi(x)$, we construct the functions $\sigma(x)$, $r_1(\la)$, and $r_2(\la)$. Next, the convergence of the series for $\sigma(x)$ and $r_1(\la)$ is proved. However, there arise difficulties with the series for $r_2(\la)$, because the obtained estimates are insufficiently precise to prove its convergence. In order to overcome these difficulties, we construct approximations $\sigma^K(x)$, $r_1^K(\la)$, and $r_2^K(\la)$ by the ``truncated'' spectral data $\{ \la_n, \alpha_n\}_{n = 1}^K$ and apply the Cauchy criterion to the sequence $\{ r_2^K\}$. Finally, a problem $L = L(\sigma, r_1, r_2)$ is constructed, and it is proved that $\{ \la_n, \alpha_n\}_{n \ge 1}$ are its spectral data.

Suppose that the boundary value problem $\tilde L$ fulfills the hypothesis of Theorem~\ref{stability_thm}. Let $\{ \la_n, \alpha_n\}_{n \ge 1}$ be arbitrary complex numbers satisfying the inequality \eqref{estde} for some $\delta_0 > 0$. In view of \eqref{estde} and \eqref{defden}, we have
\begin{equation} \label{dede}
\delta = \left( \sum_{n = 1}^{\infty} \delta_n^2 \right)^{1/2}.
\end{equation}

Using \eqref{defQ} and \eqref{Hmatr}, construct the operator $\tilde H(x)$.

\begin{lem}\label{main_equation_thm_solvability}
If $\delta_0 > 0$ is sufficiently small, then the operator $(E + \tilde H(x))$ has a bounded inverse operator in $m$ for each fixed $x \in [0,\pi]$.
Therefore, equation \eqref{main_equation} is uniquely solvable.
\end{lem}

\begin{proof}	
For operator $\tilde H(x)$, the following estimate holds (see \cite{BondTamkang}):
$$
\|\tilde H(x)\|_{m \to m} \le C \sup_n\sum_k\frac{\delta_k}{|n-k|+1}.
$$
Using the Cauchy-Bunyakovsky inequality and \eqref{dede}, we conclude that $\| \tilde H(x) \| \le C \delta$.
So, there exists $\delta_0>0$ such that, for each $\delta \le \delta_0$, we have $\|\tilde H(x)\| \le \dfrac{1}{2}$.
Indeed, since $C$ is a fixed constant, one can take $\delta_0 = \dfrac{1}{2C}$. This yields the claim.
\end{proof}

Suppose that $\delta_0$ satisfies Lemma~\ref{main_equation_thm_solvability}. Then, we can consider the solution $\psi(x)$ of the main equation \eqref{main_equation} and study its properties.

\begin{lem}\label{estimates}
Let $\psi(x)$ be the solution of the main equation \eqref{main_equation}. Then the entries of $T^{-1}\psi(x)$ can be represented in the form $\varphi_{ni}(x) = \cos(n-M_1-1)x + g_{ni}(x)$, $(n, i) \in V$, where the functions $g_{ni}(x)$ are continuous on $[0, \pi]$, the sequence $\{g_{ni}(x)\}_{(n, i) \in V}$ belongs to $l_2$ for each fixed $x \in [0, \pi]$, and $\|\{g_{ni}(x)\}\|_{l_2}$ is uniformly bounded for $x \in [0, \pi]$. Moreover,
\begin{enumerate}
\item $\|\psi(x)\| \le C$,
\item $\|g(x)\|_{l_2} \le C$,
\item $\|g(x) - \tilde g(x)\|_{l_2} \le C\delta$,
\item $\| \{g_{n0}(\pi) - g_{n1}(\pi)\} \|_{l_1} \le C\delta$,
\item $\| \{g_{n0}(x) - g_{n1}(x) - \tilde g_{n0}(x) + \tilde g_{n1}(x)\} \|_{l_1} \le C\delta$, 
\item $\| \Theta(x) \|_{L_2} \le C$,
\item $\| \Theta(x) - \tilde \Theta(x)\|_{L_2} \le C\delta$,
\end{enumerate}
where $\Theta(x) = \sum\limits_{k=1}^\infty (g_{k0} - g_{k1})\cos((k-M_1-1)x)$.
\end{lem}

\begin{proof}
By Lemma~\ref{main_equation_thm_solvability}, there exists the operator $\tilde P(x) = (E+\tilde H(x))^{-1}$, which is bounded for each fixed $x \in [0, \pi]$. In particular, $\|\tilde P(x_0)\| < \infty$, $x_0 \in [0,\pi]$. It can be shown similarly to Lemma~4.5 in \cite{BondTamkang} that $\tilde H(x)$ is continuous at $x = x_0$.
Hence, there exists $\varepsilon > 0$ such that, for each $x \in [0, \pi]$ satisfying $|x - x_0| \le \varepsilon$, we have
$$
\|\tilde H(x_0) - \tilde H(x)\| \le \frac{1}{\|\tilde P(x_0)\|}.
$$

Then, we obtain
$$
\tilde P(x) - \tilde P(x_0) = \sum\limits_{k=1}^\infty \big( \tilde H(x_0) - \tilde H(x) \big)^k \big( \tilde P(x_0) \big)^{k+1},
$$
and
$$
\|\tilde P(x) - \tilde P(x_0)\|\le2\|\tilde P(x_0)\|^2\|\tilde H(x_0) - \tilde H(x)\| \xrightarrow{x\to x_0} 0.
$$
So, $\psi(x) = \tilde P(x) \tilde \psi(x)$ is continuous by $x$. Hence $\|\psi(x)\| \le C$.

Denote
$$
g_{ni}(x) = \varphi_{ni}(x) - \cos(n-M_1-1)x, \quad g(x) = (g_n(x))_{n \ge 1}, \quad g_n(x) = 
	\begin{pmatrix}
	g_{n0}(x)\\
	g_{n1}(x)
	\end{pmatrix}.
$$

Using \eqref{defH}, we get from \eqref{main_equation} that
\begin{gather}\label{main_eq_g}
g(x) = \tilde g(x) - T^{-1}\tilde H(x)\psi(x) = \tilde g(x) - \tilde Q(x) T^{-1} \psi(x).
\end{gather}

It follows from Proposition~\ref{transform_varphi} and the asymptotics \eqref{la_asymp} that $\|\tilde g(x)\|_{l_2} \le C$. By virtue of Lemma~\ref{lem_qt}, the operator $\tilde Q(x) T^{-1}$ acts from $m$ to $l_2$ for each fixed $x \in [0,\pi]$ and $\| \tilde Q(x) T^{-1} \|_{m \to l_2} \le C$. Consequently, we obtain from \eqref{main_eq_g} and the estimate $\| \psi(x) \|_m \le C$ that $\| g(x)\|_{l_2} \le C$.

The following estimates are obtained with the same type of calculations. We demonstrate it for the estimate 4. Using \eqref{main_eq_g}, we get
\begin{align*}
\|\{ g_{n0}(\pi) - g_{n1}(\pi)\}\|_{l_1} = & \sum\limits_k|g_{k0}(\pi) - g_{k1}(\pi)| \le \|\{\tilde g_{n0}(\pi) - \tilde g_{n1}(\pi)\}\|_{l_1} \\
& + \sum\limits_k\bigg|\big(\tilde Q(\pi)T^{-1}\psi(\pi)\big)_{k1} - (\tilde Q(\pi)T^{-1}\psi(\pi)\big)_{k0}\bigg|.
\end{align*}

Using Proposition~\ref{transform_varphi} and \eqref{difcos}, we can get
\begin{equation} \label{smg}
\| \{ \tilde g_{n0}(\pi) - \tilde g_{n1}(\pi) \}\|_{l_1} = \sum\limits_k|\tilde\varphi_{k0}(\pi) - \tilde\varphi_{k1}(\pi)| = \sum_k |\zeta_k \xi_k| + O(\delta_k^2),
\end{equation}
where
$$
\xi_k := \pi\sin(\rho_{k1}\pi) + \int\limits_0^\pi {\mathcal K_1(\pi, t) t\sin(\rho_{k1}t)}\,dt.
$$
In view of the eigenvalue asymptotics \eqref{la_asymp}, we conclude that $\{ \xi_k \} \in l_2$. Therefore, using \eqref{dede}, \eqref{smg}, and $|\zeta_k| \le \delta_k$, we obtain $\| \{ \tilde g_{n0}(\pi) - \tilde g_{n1}(\pi) \}\|_{l_1} \le C \delta$.

Due to \eqref{fk}, \eqref{QF}, \eqref{estF}, and \eqref{difcos}, we get
\begin{align*}
\sum\limits_k\bigg|\big(\tilde Q(\pi)T^{-1}\psi(\pi)\big)_{k1} - (\tilde Q(\pi)T^{-1}\psi(\pi)\big)_{k0}\bigg| \le \sum\limits_k \int\limits_0^\pi|F(t)(\tilde\varphi_{k0}(t) - \tilde\varphi_{k0}(t))|dt \\
= \sum\limits_k \int\limits_0^\pi|F(t)(-\zeta_k t\sin(\rho_{k1}t) + O(\delta_k^2))|dt \le C\delta.
\end{align*}

Hence $\| \{ g_{n0}(\pi) - g_{n1}(\pi) \} \|_{l_1} \le C \delta$.
The other estimates can be proved similarly.
\end{proof}

For the next calculations, introduce the notations
\begin{align*}
&\Pi_N(\la)=\prod_{k = 1}^{N} \frac{\la - \la_{k0}}{\la - \la_{k1}},\\
&\mathcal S_N(\la) = \sum_{k=1}^N \sum_{j=0}^1 \frac{(-1)^j\alpha_{kj}\varphi_{kj}(\pi)(\tilde r_1(\la)\tilde\varphi^{[1]}_{kj}(\pi) + \tilde r_2(\la)\tilde\varphi_{kj}(\pi))}{\la - \la_{kj}},
\end{align*}
where $N \in \mathbb N \cup \{ \infty \}$.

Using the functions $\{ \varphi_{ni}(x)\}_{(n,i) \in V}$ of Lemma~\ref{estimates}, we construct $\sigma(x)$ and $r_1(\lambda)$ as the series \eqref{sigma_series} and \eqref{r1_series}, respectively. Lemmas~\ref{lem:conv} and \eqref{lem:sigma:r1} establish the convergence of these series.

\begin{lem} \label{lem:conv}
The series \eqref{sigma_series} converges  in $L_2(0, \pi)$. Moreover, $\| \sigma(x) - \tilde \sigma(x) \|_{L_2} \le C \delta$.
\end{lem}

\begin{proof}
The convergence can be proved in the same way as Lemma~5.4 in \cite{BondTamkang} by using the estimates 2, 3, 5--7 from Lemma~\ref{estimates}. 
Therefore, let us obtain the estimate. In view of \eqref{sigma_series}, we have $|\sigma(x) - \tilde\sigma(x)| = 2|S(x)|$, where
$$
S(x) = \sum\limits_{k=1}^\infty \bigg( \alpha_{k0}\tilde\varphi_{k0}(x)\varphi_{k0}(x) - \alpha_{k1}\tilde\varphi_{k1}(x)\varphi_{k1}(x) - \frac{1}{2}(\alpha_{k0} - \alpha_{k1})  \bigg).
$$
We can show that $S(x) = S_1(x) + S_2(x) + S_3(x) + S_4(x)$, where
\begin{align*}
&S_1(x) = \sum\limits_{k=1}^\infty (\alpha_{k0} - \alpha_{k1})\bigg(\cos^2(k-M_1-1)x - \frac{1}{2}\bigg),\\
&S_2(x) = \sum\limits_{k=1}^\infty (\alpha_{k0}\tilde g_{k0}(x)g_{k0}(x) - \alpha_{k1}\tilde g_{k1}(x)g_{k1}(x)), \\
&S_3(x) = \sum\limits_{k=1}^\infty \alpha_{k0}(\tilde g_{k0}(x) + g_{k0}(x) - \tilde g_{k1}(x) - g_{k1}(x))\cos(k-M_1-1)x,\\
&S_4(x) = \sum\limits_{k=1}^\infty (\alpha_{k0} - \alpha_{k1})(\tilde g_{k1}(x) + g_{k1}(x))\cos(k-M_1-1)x.
\end{align*}

Using the estimates 2, 3, 5--7 of Lemma~\ref{estimates}, we can get that $\|S_j\|_{L_2} \le C\delta$ for $x \in [0, \pi]$. This yields the claim.
\end{proof}

\begin{lem}\label{lem:sigma:r1}
The infinite product $\Pi_{\infty}(\la)$ and the series $S_{\infty}(\la)$ in \eqref{r1_series} converge absolutely and uniformly on compact sets $|\la| \le C$, $|\la - \la_{ni}| \ge \varepsilon$, $n \ge 1$, $i = 0,1$, $C, \varepsilon > 0$. Moreover, $|r_1(\la) - \tilde r_1(\la)| \le C\delta$.
\end{lem}

\begin{proof}
Let us start from the infinite product. Due to the asymptotics \eqref{la_asymp}, we can get:
$$
\bigg|1 + \frac{\la_{k1} - \la_{k0}}{\la - \la_{k1}}\bigg| = \bigg|1 + \frac{\zeta_k^2 - 2\rho_{k1}\zeta_k}{\la - \la_{k1}}\bigg| \le 1 + d_k(\la), \quad d_k(\la) := \frac{2|\rho_{k1}|\delta_k + \delta_k^2}{|\la - \la_{k1}|},
$$
so
$$
|\Pi_\infty(\la)| = \bigg|\prod_{k = 1}^{\infty} \bigg(1 + \frac{\la_{k1} - \la_{k0}}{\la - \la_{k1}}\bigg)\bigg| \le \prod_{k = 1}^{\infty} \left(1 + d_k(\la) \right).
$$

Suppose that $|\la| \le C$, $|\la - \la_{ki}| \ge \varepsilon$ for all $k \ge 1$, $i= 0, 1$. Then, in view of the asymptotics \eqref{la_asymp},
$|\la - \la_{k1}| \ge Ck^2$ and $2|\rho_{k1}| + \delta_k \le Ck$, because $\delta_k \le \delta$. Hence
\begin{gather} \label{dk}
d_k(\la) \le \frac{C\delta_k}{k},
\end{gather}
so the infinite product $\Pi_{\infty}(\la)$ converges uniformly by $\la$.

Let us obtain an estimate for this product. Taking the  logarithm, we get the series:
$$
\ln \prod_{k = 1}^{\infty} (1 + d_k(\la)) = \sum\limits_{k=1}^\infty \ln(1 + d_k(\la)).
$$
By the complex Taylor theorem, we obtain
$$
\ln(1 + d_k(\la)) = d_k(\la)\frac{1}{2\pi i}\oint_\gamma \frac{\ln(1+\omega)}{\omega(\omega + d_k(\la))}d\omega = R_0^k(\la),
$$
where $\gamma = re^{it}$, $r > 0$ is a fixed sufficiently small radius, $t \in [0, 2\pi]$. 
Note that
$$
\max\limits_{\omega \in \gamma}\bigg|\dfrac{\ln(1+\omega)}{\omega(\omega + d_k)}\bigg| \le C.
$$
So, due to \eqref{dk} and \eqref{dede}, we get:
$$
\sum\limits_{k=1}^\infty |\ln(1 + d_k(\la))| = \sum\limits_{k=1}^\infty |R_0^k(\la)| \le \sum\limits_{k=1}^\infty \frac{\delta_k}{k} \le C \delta.
$$
Thus $|\ln\Pi_\infty(\la)| \le C \delta$.
As
$$
\Pi_\infty(\la) = e^{\ln\Pi_\infty(\la)} = 1 + O(\ln\Pi_\infty(\la)),
$$
then
\begin{gather}\label{prod_est}
|\Pi_\infty(\la) - 1| \le C|\ln\Pi_\infty(\la)| \le C\delta.
\end{gather}

Now let us consider the series $\mathcal S_\infty(\la)$. Using the estimates $|\la| \le C$, $|\la - \la_{kj}| \ge Ck^2$, $j = 0, 1$, we can get:
$$
|\mathcal S_\infty(\la)| \le \sum\limits_{i=0}^1\sum\limits_{j=0}^1\mathcal S^{ij},
$$
where
\begin{align*}
&\mathcal S^{ij} = \sum_{k=1}^\infty \frac{\big|\la_{k0}^i\alpha_{k1}\varphi_{k1}(\pi)\tilde\varphi^{[j]}_{k1}(\pi) - \la_{k1}^i\alpha_{k0}\varphi_{k0}(\pi)\tilde\varphi^{[j]}_{k0}(\pi)\big|}{k^4}.
\end{align*}

Let us consider $\mathcal S^{11}$. The other terms can be treated in the same way. Using simple calculations, it can be shown that 
$$
\mathcal S^{11} \le \mathcal J^1 + \mathcal J^2 + \mathcal J^3 + \mathcal J^4,
$$
where
\begin{align*}
&\mathcal J^1 = \sum_k\frac{C}{k^4}|\la_{k1} - \la_{k0}||\alpha_{k0}\varphi_{k0}(\pi)\tilde\varphi^{[1]}_{k0}(\pi)|,\\
&\mathcal J^2 = \sum_k\frac{C}{k^4}|\la_{k0}||\alpha_{k0} - \alpha_{k1}||\varphi_{k0}(\pi)\tilde\varphi^{[1]}_{k0}(\pi)|,\\
&\mathcal J^3 = \sum_k\frac{C}{k^4}|\la_{k0}\alpha_{k1}||\varphi_{k0}(\pi) - \varphi_{k1}(\pi)||\tilde\varphi^{[1]}_{k0}(\pi)|,\\
&\mathcal J^4 = \sum_k\frac{C}{k^4}|\la_{k0}\alpha_{k1}\varphi_{k1}(\pi)|\tilde\varphi^{[1]}_{k0}(\pi) - \tilde\varphi^{[1]}_{k1}(\pi)|.
\end{align*}

The formulas \eqref{la_asymp}, \eqref{alpha_asymp}, \eqref{defden}, Proposition~\ref{transform_varphi}, and Lemma~\ref{estimates} imply the following estimates: 
$$
|\la_{k1} - \la_{k0}| \le Ck\delta_k, \quad |\alpha_{k0}| \le C, \quad |\varphi_{k0}(\pi)| \le C, \quad |\tilde \varphi_{k0}^{[1]}(\pi)| \le Ck\tau_k,
$$
where $\{ \tau_k\} \in l_2$.
Therefore, we have
$$
\mathcal J^1 \le \sum_k\frac{C\delta_k \tau_k}{k^2} \le C \delta.
$$

In the same way, using the estimate 4 from Lemma~\ref{estimates} and Proposition~\ref{transform_varphi}, we can get that $\mathcal J^j \le C\delta$, $j=\overline{1, 4}$, so $\mathcal S^{11} \le C\delta$.
Consequently, the series $\mathcal S_\infty(\la)$ converges on compact sets excluding the eigenvalues and 
\begin{gather}\label{sum_r1_est}
|\mathcal S_\infty(\la)| \le C\delta.
\end{gather}

Now let us get the estimate for $|r_1(\la) - \tilde r_1(\la)|$. It is follows from \eqref{r1_series} that
\begin{gather}\label{r1_short}
r_1(\la) = \Pi_\infty(\la)(\tilde r_1(\la) - \mathcal S_\infty(\la)).
\end{gather}

So,
\begin{align*}
|r_1(\la) - \tilde r_1(\la)| = |\Pi_\infty(\la)(\tilde r_1(\la) - \mathcal S_\infty(\la)) - \tilde r_1(\la)| \le |\Pi_\infty(\la) - 1| |\tilde r_1(\la) - \mathcal S_\infty(\la)| + |\mathcal S_\infty(\la)|.
\end{align*}
Substitung \eqref{prod_est} and \eqref{sum_r1_est} we get the required estimate.
\end{proof}

Thus, we have proved the convergence in the reconstruction formula~\eqref{r1_series} for $r_1(\lambda)$. However, we cannot directly obtain the similar results for $r_2(\lambda)$, since the formula \eqref{r2_series} contains the quasi-derivatives $\varphi^{[1]}_{kj}(\pi)$, for which we have no sufficiently precise estimates. In order to overcome this difficulty, we use the approximation approach.

For sufficiently large $K \in \mathbb N$, consider the following data $\{\la_n^K, \alpha_n^K\}_{n \ge 1}$:
\begin{gather} \label{trunc}
\la_n^K =
    \left\{\begin{array}{ll}
        \la_n, \quad & n \le K,\\
        \tilde \la_n, \quad & n > K,\\
    \end{array}\right.
    ,\quad
\alpha_n^K =
    \left\{\begin{array}{ll}
        \alpha_n, \quad & n \le K,\\
        \tilde \alpha_n, \quad & n > K.\\
    \end{array}\right.
\end{gather}

Similarly to $\tilde H(x)$, we introduce the operator $\tilde H^K(x) = (\tilde H^K_{u,v}(x))_{u,v \in V}$:
\begin{gather*}
\tilde H^K_{ni, kj}(x) = 
\left\{\begin{array}{ll}
        0, \quad & k > K,\\
        \tilde H_{ni, kj}(x), \quad & k \le K.\\
    \end{array}\right.
\end{gather*}

Relying on Lemma~\ref{main_equation_thm_solvability}, one can show that the operator $(E + \tilde H^K(x))$ has a bounded inverse for each fixed $x \in [0,\pi]$ if $K$ is sufficiently large. Let us pass to a finite system in the main equation \eqref{main_equation}. The equation
$$
\tilde \psi_{ni}(x) = \psi^K_{ni}(x) + \sum_{k=1}^{K} (\tilde H_{ni, k0}(x)\psi^K_{k0}(x) + \tilde H_{ni, k1}(x)\psi^K_{k1}(x)), \quad n = \overline{1, K}, i=0,1.
$$
is uniquely solvable for each fixed $x \in [0,\pi]$ due to the inversibility of operator $(E + \tilde H^K(x))$. 
Using $T^{-1}$, we can get
\begin{gather*}
\varphi^K_{n0}(x) = \delta_n\psi^K_{n0}(x)+\psi^K_{n1}(x), \quad
\varphi^K_{n1}(x) = \psi^K_{n1}(x), \quad n = \overline{1,K}.
\end{gather*}
So, we can obtain the relation similar to \eqref{series_varphi}:
\begin{equation} \notag
\varphi^K_{ni}(x) = \tilde \varphi_{ni}(x) - \sum_{k=1}^K (\tilde Q_{ni; k0}(x)\varphi^K_{k0}(x) - \tilde Q_{ni; k1}(x)\varphi^K_{k1}(x)), \quad n = \overline{1, K}, \quad i=0, 1.
\end{equation}

Let us construct the following finite approximation for $\sigma(x)$:
\begin{align}  \label{sigmaK}
	\sigma^K(x) = \tilde \sigma(x) - 2 \sum_{k=1}^K \sum_{j=0}^1 {(-1)^j\alpha_{kj}\Bigg(\tilde \varphi_{kj}(x)\varphi^K_{kj}(x) - \frac{1}{2}\Bigg)},
\end{align}
and define the corresponding quasi-derivative as $y^{[1]}(x) = y'(x) - \sigma^K(x)y(x)$.

For the next calculations, let us introduce the following notations:
\begin{align*}
&\mathcal S_N^K(\la) = \sum_{k=1}^N \sum_{j=0}^1 \frac{(-1)^j\alpha_{kj}\varphi^K_{kj}(\pi)(\tilde r_1(\la)\tilde\varphi^{[1]}_{kj}(\pi) + \tilde r_2(\la)\tilde\varphi_{kj}(\pi))}{\la - \la_{kj}},\\
&\mathcal U_N^K(\la) = - \tilde r_1(\la)\sum_{k=1}^N \sum_{j=0}^1 (-1)^j\alpha_{kj}\bigg(\tilde \varphi_{kj}(x)\varphi^K_{kj}(x) - 1\bigg)\\
& + \sum_{k = 1}^N \sum_{j=0}^1 \frac{(-1)^j\alpha_{kj}\varphi^{K[1]}_{kj}(\pi)(\tilde r_1(\la)\tilde\varphi^{[1]}_{kj}(\pi) + \tilde r_2(\la)\tilde\varphi_{kj}(\pi))}{\la - \la_{kj}}.
\end{align*}

Now let us consider finite approximations for $r_1(\la)$ and $r_2(\la)$:
\begin{align} \label{r1K}
	r_1^K(\la) & = \Pi_K(\la)(\tilde r_1(\la) - \mathcal S_K^K(\la)), \\ \label{r2K}
	r_2^K(\la) &= \Pi_K(\la)(\tilde r_2(\la) + \mathcal U_N^K(\la)).
\end{align}

\begin{lem} \label{lem:degree}
$r_1^K(\la)$ is a polynomial of degree $M_1$ with the leading coefficient $1$, and $r_2^K(\la)$ is a polynomial of degree not greater than $M_1$. In other words, $(r_1^K, r_2^K) \in \mathcal R$.
\end{lem}
\begin{proof}
Consider the series for $r_1^K(\la)$. Reduction to a common denominator implies
$$
r_1^K(\la) = \frac{\tilde r_1(\la) \prod\limits_{k=1}^K(\la-\la_{k0})(\la-\la_{k1}) - \sum\limits_{k=1}^{2K-1}\la^k(s_k\tilde r_1(\la)+t_k\tilde r_2(\la))}{\prod\limits_{k=1}^K(\la-\la_{k1})^2},
$$
where $t_k$, $s_k \in \mathbb C$.
As $(\tilde r_1, \tilde r_2) \in \mathcal R$, then this expression can be rewritten in the form:
$$
r_1^K(\la) = \frac{\sum\limits_{k=1}^{2K+M_1}\la^ks^*_k}{\prod\limits_{k=1}^K(\la-\la_{k1})^2},
$$
where $s^*_k \in \mathbb C$ and $s^*_{2K+M_1} = 1$. Due to Lemma~\ref{lem:cz}, we conclude that the fraction can be reduced and we get a polynomial of degree
$$
\mbox{deg}\,(r_1^K(\la)) = 2K+M_1-2K = M_1
$$
with the leading coefficient $1$.

In the same way we can prove that $\mbox{deg}\,(r_2^K(\la)) \le M_1$. So, $(r_1^K, r_2^K) \in \mathcal R$.
\end{proof}

Consider the boundary value problem $L^K := L(\sigma^K, r_1^K, r_2^K)$. Introducing the differential expression $\ell^Ky := -(y^{[1]})' - \sigma^K(x)y^{[1]} - (\sigma^K(x))^2y$, the problem $L^K$ can be represented as follows:
$$
	\ell^Ky = \la y, \quad x \in (0, \pi),
$$
$$
	y^{[1]}(0) = 0, \quad r_1^K(\la)y^{[1]}(\pi) + r_2^K(\la) y(\pi) = 0.
$$

\begin{lem} \label{lem:sol}
For $n \ge 1$, $i = 0,1$, the following relations hold:
\begin{gather*}
\ell^K\varphi^K_{ni}(x) = \la \varphi^K_{ni}(x), \quad x \in (0,\pi), \\
\varphi^K_{ni}(0) = 1, \quad \varphi^{K[1]}_{ni}(0) = 0.
\end{gather*}

Moreover, $\{\la^K_n, \alpha^K_n\}_{n \ge 1}$ are the spectral data for the boundary problem $L^K$.
\end{lem}

This lemma is proved by direct calculations.

\begin{lem} \label{coeffcon}
$\sigma^K(x) \to \sigma(x)$ in $L_2(0, \pi)$ and	
$r^K_1(\la) \to r_1(\la)$ uniformly on compact sets as $K \to \infty$, where $\sigma(x)$ and $r_1(\la)$ are defined by \eqref{sigma_series} and \eqref{r1_series}, respectively.
\end{lem}

\begin{proof}
The convergence $\sigma^K(x) \to \sigma(x)$ in $L_2(0, \pi)$ can be proved similarly to \cite{ChitBond}. 
Let us introduce the following vector function $\theta^K(x) = (\theta^K_{ni}(x))_{(n,i)\in V}$:
\begin{gather*}
\theta_{ni}^K(x) = \left\{\begin{array}{ll}
\varphi_{ni}(x) - \varphi_{ni}^K(x), \quad & n \le K, \\
0, \quad &  n > K.
\end{array}\right.
\end{gather*}

Analogously to \cite{BondTamkang}, we prove that
\begin{equation} \label{thetaK}
\lim_{K \to \infty}\|\theta^K(x)\|_{l_2} = 0, \quad \lim_{K \to \infty} \| \{ \theta_{n0}^K(x) - \theta_{n1}^K(x) \} \|_{l_1} = 0, 
\end{equation}
uniformly by $x \in [0, \pi]$.

Now let us consider the difference using \eqref{r1_short} and \eqref{r1K}:
\begin{align*}
|r_1(\la) - r_1^K(\la)| = & | \Pi_K(\la)| \Bigg| \frac{\Pi_\infty(\la)}{\Pi_K(\la)} ( \tilde r_1(\la) - \mathcal S_\infty(\la) )  - \tilde r_1(\la) + \mathcal S_K^K(\la)\Bigg|.
\end{align*}

Passing to the limit as $K \to \infty$ and taking the convergence of the infinite product $\Pi_\infty(\la)$ into account, we obtain
\begin{align*}
\lim\limits_{K \to \infty }|r_1(\la) - r_1^K(\la)| = & g(\la) \lim\limits_{K \to \infty} \Bigg| \sum\limits_{k=1}^K \sum\limits_{j=0}^1 \frac{(-1)^j\alpha_{kj}\theta^K_{kj}(\pi)\big(\tilde r_1(\la)\tilde\varphi^{[1]}_{kj}(\pi) + \tilde r_2(\la)\tilde\varphi_{kj}(\pi) \big)}{\la - \la_{kj}} \\
& + \mathcal S_\infty(\la) - \mathcal S_K(\la)\Bigg|,
\end{align*}
where $g(\la) = \Bigg| \lim\limits_{K\to \infty} \Pi_K(\la) \Bigg| < \infty$ on the compact sets $|\la| \le C$, $|\la - \la_{k1}| \ge \varepsilon$.
Due to \eqref{thetaK} and Lemma~\ref{lem:sigma:r1}, we obtain that
$$
\lim_{K \to \infty}|r_1(\la) - r_1^K(\la)| = 0.
$$
\end{proof}

For $r_2(\la)$, we have no analog of Lemma~\ref{lem:sigma:r1}. Therefore, we prove that the convergence of the polynomial sequence $\{ r_2^K(\la)\}$ by using the Cauchy criterion. 

\begin{lem}\label{lem:fundamental}
$\{r_2^k(\la)\}$ is a Cauchy sequence in the sense of the uniform convergence by $\la$ on compact sets.
\end{lem}
\begin{proof}
Consider two elements $r_2^K(\la)$ and $r_2^N(\la)$ of the sequence $\{r_2^k(\la)\}$, where $K$ and $N$ are sufficiently large integer numbers. Without loss of generality we suppose that $K > N$.
Using \eqref{r2K}, consider the difference
$$
r_2^K(\la) - r_2^N(\la) = \Pi_N(\la) \bigg(\frac{\Pi_K(\la)}{\Pi_N(\la)}\bigg(\tilde r_2(\la)+\mathcal U_K^K(\la)\bigg) - \tilde r_2(\la) -\mathcal U_N^N(\la)\bigg).
$$

By Lemma~\ref{lem:sigma:r1}, $\Pi_N(\la) \to \Pi_{\infty}(\la)$ as $N \to \infty$ uniformly by $\la$ on compact sets excluding the eigenvalues. Hence $| \Pi_N(\la) | \le C$ for $\la$ on such a set.
Then, we get the estimate
\begin{equation} \label{estr2}
|r_2^K(\la) - r_2^N(\la)| \le C\bigg|\frac{\Pi_K(\la)}{\Pi_N(\la)} - 1\bigg||S_1^K(\la)| + C|S^{K,N}(\la)|,
\end{equation}
where
\begin{align*}
S_1^K(\la) = & \tilde r_2(\la)+\mathcal U_K^K(\la), \\
S^{K,N}(\la) = &\tilde r_1(\la)\sum\limits_{k=1}^N\sum\limits_{j=0}^1(-1)^j\alpha_{kj}\tilde\varphi_{kj}(\pi)\big(\varphi_{kj}^N(\pi) - \varphi_{kj}^K(\pi) \big) \\
& + \sum\limits_{k=1}^N\sum\limits_{j=0}^1 \frac{(-1)^j\alpha_{kj}\big(\varphi_{kj}^K(\pi) - \varphi_{kj}^N(\pi) \big)(\tilde r_1(\la)\tilde\varphi^{[1]}_{kj}(\pi) + \tilde r_2(\la)\tilde\varphi_{kj}(\pi))}{\la - \la_{kj}} + \mathcal U_K^K(\la) - \mathcal U_N^K(\la).
\end{align*}

By virtue of Lemma~\ref{lem:sol}, the functions $\varphi_{kj}^K(x)$ and $\varphi_{kj}^N(x)$ satisfy the equations analogous to \eqref{eqv} with the coefficients $\sigma^K(x)$ and $\sigma^N(x)$, respectively. Consequently, Proposition~\ref{transform_varphi} implies
\begin{gather}\label{varphiK}
	\varphi^K_{kj}(\pi) = \cos(\rho_{kj}\pi) + \int_0^{\pi}  \mathcal{K}^K_1(\pi, t) \cos(\rho_{kj} t) \, dt, \\ \label{der_varphiK}
	\varphi^{K[1]}_{kj}(\pi) = -\rho_{kj}\sin (\rho_{kj} \pi) + \rho\int_0^{\pi} \mathcal{K}^K_2(\pi, t) \sin(\rho_{kj} t) \, dt + \mathcal{C}^K(\pi),\\ \label{varphiN}
	\varphi^N_{kj}(\pi) = \cos(\rho_{kj}\pi) + \int_0^{\pi}  \mathcal{K}^N_1(\pi, t) \cos(\rho_{kj} t) \, dt, \\ \label{der_varphiN}
	\varphi^{N[1]}_{kj}(\pi) = -\rho_{kj}\sin(\rho_{kj} \pi) + \rho\int_0^{\pi} \mathcal{K}^N_2(\pi, t) \sin(\rho_{kj} t) \, dt + \mathcal{C}^N(\pi),
\end{gather}
where the functions $\mathcal K_j^K(\pi, t)$ and $\mathcal K_j^N(\pi, t)$ for $j = 1, 2$ belong to $L_2(0,\pi)$ and $\mathcal C^K(\pi)$, $\mathcal C^N(\pi)$ are constants. By virtue of Lemma~\ref{coeffcon}, $\sigma^K(x) \to \sigma(x)$ in $L_2(0,\pi)$ as $K \to \infty$. Consequently, one can show that $\mathcal K_j^K(\pi, t) \to \mathcal K_j(\pi, t)$ in $L_2(0,\pi)$ and $\mathcal C^K(\pi) \to \mathcal C(\pi)$ as $K \to \infty$. Hence, $\{ \mathcal K_j^K(\pi, t)\}$ and $\{ \mathcal C^K(\pi)\}$ are Cauchy sequences, so
$$
\lim_{K,N \to \infty} \| \mathcal K_j^K(\pi, t) - \mathcal K_j^N(\pi, t) \|_{L_2} = 0, \quad j = 1,2, \qquad \lim_{K,N \to \infty} |\mathcal C^K(\pi) - \mathcal C^N(\pi)| = 0.
$$

Therefore, using the representations \eqref{varphiK}--\eqref{der_varphiN} and Lemma~\ref{estimates}, it can be shown that 
\begin{equation} \label{SKN}
\lim_{K,N \to \infty} S^{K,N}(\la) = 0
\end{equation}
uniformly on compact sets excluding the eigenvalues.

Using \eqref{der_varphiK} and the estimate 4 from Lemma~\ref{estimates} analogously to the Lemma~\ref{lem:sigma:r1}, we can show that the series in $S_1^K(\la)$ converges absolutely and uniformly by $\la$ on compact sets excluding the eigenvalues. Hence $\| S_1^K(\la) \| \le C$.

Since the sequence $\{ \Pi_N(\la)\}$ converges, then it is a Cauchy sequence, so
$$
\lim\limits_{K, N \to \infty}\bigg|\frac{\Pi_K(\la)}{\Pi_N(\la)} - 1\bigg| = \lim\limits_{K, N \to \infty}\bigg|\frac{\Pi_K(\la) - \Pi_N(\la)}{\Pi_N(\la)}\bigg| = 0.
$$
Together with \eqref{estr2} and \eqref{SKN}, this yields 
$$
\lim_{K,N \to \infty} |r_2^K(\la) - r_2^N(\la) | = 0
$$
uniformly by $\la$ on compact sets excluding the eigenvalues. However, since $\{ r_2^K(\la) \}$ are polynomials, then the maximum modulus principle implies the uniform convergence on any compact sets. This concludes the proof.
\end{proof}

\begin{proof}[Proof of Theorem~\ref{stability_thm}]
Let $\tilde L$ be a boundary value problem $L(\tilde \sigma, \tilde r_1, \tilde r_2)$ with $\tilde \sigma \in L_2(0,\pi)$, $(\tilde r_1, \tilde r_2) \in \mathcal R$, and simple eigenvalues $\{ \tilde \la_n\}_{n \ge 1}$. Consider arbitrary data $\{ \la_n, \alpha_n\}_{n \ge 1}$ satisfying the estimate \eqref{estde} for sufficiently small $\delta_0 > 0$. Construct the main equation \eqref{main_equation}. By virtue of Lemma~\ref{main_equation_thm_solvability}, the main equation has the unique solution $\psi(x)$. 

Next, find the entries $\varphi_{ni}(x)$, $(n,i) \in V$, of the vector $T^{-1} \psi(x)$ and construct the functions $\sigma(x)$ and $r_1(\la)$ by the formulas \eqref{sigma_series} and \eqref{r1_series}, respectively. Lemmas~\ref{lem:conv} and~\ref{lem:sigma:r1} assert the convergence of the series for $\sigma(x)$ in $L_2(0,\pi)$ and for $r_1(\la)$ absolutely and uniformly on compact sets excluding the eigenvalues $\{ \la_{ni}\}_{(n,i) \in V}$. Moreover, these lemmas imply the estimates $\| \sigma(x) - \tilde \sigma(x) \|_{L_2} \le C \delta$ and $| r_1(\la) - \tilde r_1(\la) | \le C \delta$ for $\la$ on compact sets. 

In order to deal with $r_2(\la)$, we use the approximation approach. For sufficiently large $K$, we consider the ``truncated'' spectral data $\{ \la_n^K, \alpha_n^K\}_{n \ge 1}$ defined by \eqref{trunc}. Then, by using these data, we construct the finite approximations $\sigma^K(x)$, $r_1^K(\la)$, and $r_2^K(\la)$ by the formulas \eqref{sigmaK}, \eqref{r1K}, and \eqref{r2K}, respectively. Lemma~\ref{lem:degree} asserts that $(r_1^K, r_2^K) \in \mathcal R$. Consider the boundary value problem $L^K = L(\sigma^K, r_1^K, r_2^K)$. By Lemma~\ref{lem:sol}, $\{ \la_n^K, \alpha_n^K\}_{n \ge 1}$ are the spectral data of $L^K$.

Next, we pass to the limit as $K \to \infty$. Lemma~\ref{coeffcon} implies that $\sigma^K(x) \to \sigma(x)$ in $L_2(0,\pi)$ and $r_1^K(\la) \to r_1(\la)$ on compact sets. By Lemma~\ref{lem:fundamental}, $\{ r_2^K(\la)\}_{K \ge 1}$ is a Cauchy sequence. Hence, this sequence converges to some $r_2^*(\la)$ uniformly on compact sets. Since $(r_1^K, r_2^K) \in \mathcal R$ for all sufficiently large $K$, the uniform convergence implies that $r_1(\la)$ and $r_2^*(\la)$ are also polynomials and $(r_1, r_2^*) \in \mathcal R$. Consider the boundary value problem $L := L(\sigma, r_1, r_2^*)$. Since $\sigma^K \to \sigma$, $r_1^K \to r_1$, $r_2^K \to r_2$ in a suitable sense, then Lemma~5.9 from \cite{ChitBond} implies that the spectral data of $L^K$ converge to the spectral data of $L$. In view of \eqref{trunc}, we have $\la_n^K \to \la_n$ and $\alpha_n^K \to \alpha_n$ as $K \to \infty$ for each fixed $n$. Hence, $\{ \la_n, \alpha_n \}_{n \ge 1}$ are the spectral data of $L$. 

The estimates \eqref{stab} for $\sigma(x)$ and $r_1(\la)$ are already proved. It remains to obtain the estimate $|r_2(\la) - \tilde r_2(\la)| \le C \delta$ for $\la$ on compact sets. For this purpose, we use the formula \eqref{r2_series}. Now, we know that the entries $\varphi_{ni}(x)$, which are obtained from the solution of the main equation, satisfy \eqref{eqv} with $\la = \la_{ni}$, so Proposition~\ref{transform_varphi} holds for $\varphi_{ni}(\pi)$ and $\varphi^{[1]}_{ni}(\pi)$. Using Proposition~\ref{transform_varphi} and following the proofs of Lemmas~\ref{lem:sigma:r1} and \ref{coeffcon}, one can show that the series in \eqref{r2_series} converge absolutely and uniformly on compact sets excluding eigenvalues, $|r_2(\la) - \tilde r_2(\la)| \le C \delta$ on compact sets, and $r_2^K(\la) \to r_2(\la)$ as $K \to \infty$. The latter convergence implies $r_2(\la) = r_2^*(\la)$ and so completes the proof.
\end{proof}

\begin{remark}
The absolute and uniform convergence of the polynomials $\{ r_1^K(\la)\}$ and $\{ r_2^K(\la)\}$ on compact sets implies the convergence for the corresponding coefficients:
\begin{align*}
r_1^K(\la) = \la^{M_1} + \sum_{n = 0}^{M_1-1} c_n^K \la^n, \quad
& r_2^K(\la) = \sum_{n = 0}^{M_1} d_n^K \la^n, \\
\lim_{K \to \infty} c_n^K = c_n, \quad & \lim_{K \to \infty} d_n^K = d_n.
\end{align*}
\end{remark}

\section{Multiple eigenvalues} \label{sec:mult}

In general, the eigenvalues of the boundary value problem \eqref{eqv}-\eqref{bc} are not necessarily simple. Due to the asymptotics \eqref{la_asymp}, the number of multiple eigenvalues is finite.
In this section, we consider a special case, when multiple eigenvalues are possible but they are not perturbed. Theorems~\ref{stability_thm} and \ref{mainthm} can be extended to this case with minor technical modification. However, the general case of splitting for multiple eigenvalues requires a separate investigation. 

Let $m_k$ be the multiplicity of $\la_k$. Put $I := \{1\}\cup\{n > 1 \colon \la_{n-1} \neq \la_n\}$. Then, the Weyl function admits the following representation (see \cite{ChitBond}):
\begin{gather}\label{weyl_sum}
	M(\la) = \sum_{k \in I}^{} \sum_{j=0}^{m_k - 1} \frac{\alpha_{k+j}}{(\la - \la_k)^{j+1}},
\end{gather}
where $\alpha_{k+j}$ are the coefficients of the principal part of the Weyl function $M(\la)$ in a neighbourhood of the pole $\la_k$. Thus, the spectral data of the problem $L$ are denoted as $\{ \la_n, \alpha_n\}_{n \ge 1}$ as before and the statement of Inverse Problem~\ref{ip:main3} remains the same.

Let $L = L(\sigma, r_1, r_2)$ and $\tilde L = L(\tilde \sigma, \tilde r_1, \tilde r_2)$ be two boundary value problems such that
$$
\la_n = \tilde \la_n, \quad \alpha_n = \tilde \alpha_n, \quad n \le N,
$$
where $N$ is such an index that all the eigenvalues $\la_n$ and $\tilde \la_n$ for $n > N$ are simple.
In this case, we get the following reconstruction formulas:
\begin{align} \label{sigma_series_m}
\sigma(x) = \tilde \sigma(x) - 2 \sum_{k=N+1}^\infty \sum_{j=0}^1 {(-1)^j\alpha_{kj}\Bigg(\tilde \varphi_{kj}(x)\varphi_{kj}(x) - \frac{1}{2}\Bigg)},
\end{align}
\begin{align} \label{r1_series_m}
r_1(\la) = \prod_{k = N+1}^{\infty} \frac{\la - \la_{k0}}{\la - \la_{k1}}\Bigg(\tilde r_1(\la) - \sum_{k=N+1}^\infty \sum_{j=0}^1 \frac{(-1)^j\alpha_{kj}\varphi_{kj}(\pi)(\tilde r_1(\la)\tilde\varphi^{[1]}_{kj}(\pi) + \tilde r_2(\la)\tilde\varphi_{kj}(\pi))}{\la - \la_{kj}} \Bigg),
\end{align}
\begin{align}\notag 
r_2(\la) &= \prod_{k = N+1}^{\infty} \frac{\la - \la_{k0}}{\la - \la_{k1}}\Bigg(\tilde r_2(\la) - \tilde r_1(\la) \sum_{k=N+1}^\infty \sum_{j=0}^1 {(-1)^j\alpha_{kj}\Bigg(\tilde \varphi_{kj}(\pi)\varphi_{kj}(\pi) - 1\Bigg)} \\ \label{r2_series_m}
&+ \sum_{k = N+1}^{\infty} \sum_{j=0}^1 \frac{(-1)^j\alpha_{kj}\varphi^{[1]}_{kj}(\pi)(\tilde r_1(\la)\tilde\varphi^{[1]}_{kj}(\pi) + \tilde r_2(\la)\tilde\varphi_{kj}(\pi))}{\la - \la_{kj}} \Bigg),
\end{align}
where the series in \eqref{sigma_series_m} converges in $L_2(0,\pi)$ and the series and the infinite products in \eqref{r1_series_m} and \eqref{r2_series_m} converge absolutely and uniformly on compact sets in $\mathbb C \setminus \{ \la_{ni}\}_{n \ge N+1, \, i = 0, 1}$.

Furthermore, the following theorem on local solvability and stability is valid.

\begin{thm}\label{stability_thm_m}
	Let $\tilde L = L(\tilde \sigma, \tilde r_1, \tilde r_2)$ be a fixed boundary value problem of form \eqref{eqv}-\eqref{bc} with $\tilde \sigma \in L_2(0,\pi)$ and $(\tilde r_1, \tilde r_2) \in \mathcal R$. Let the eigenvalues $\tilde \la_n$ with numbers $n > N$ be simple, and for $n \le N$ they can be multiple.
	Then, there exists $\delta_0 > 0$, depending on problem $\tilde L$, such that, for any complex numbers $\{\la_n, \alpha_n\}_{n \ge 1}$ satisfying the conditions
	\begin{gather*}
	\la_n = \tilde \la_n, \quad \alpha_n = \tilde \alpha_n, \quad n \le N, \\
	\delta := \bigg(\sum\limits_{n=N+1}^\infty(|\tilde \rho_n - \rho_n|+|\tilde\alpha_n - \alpha_n|)^2\bigg)^{\frac{1}{2}} \le \delta_0,
	\end{gather*}
	there exist a complex-valued function $\sigma(x) \in L_2(0, \pi)$ and polynomials $(r_1, r_2)\in\mathcal R$, which are the solution of Inverse Problem~\ref{ip:main3} for the data $\{\la_n, \alpha_n\}_{n \ge 1}$. Moreover,
	\begin{equation*} 
	\|\sigma - \tilde\sigma\|_{L_2} \le C\delta, \quad |r_1(\la) - \tilde r_1(\la)| \le C\delta, \quad |r_2(\la) - \tilde r_2(\la)| \le C\delta
	\end{equation*}
	for $\la$ on compact sets, where the constant $C$ depends only on $\tilde L$, $\delta_0$, and on a compact set in the last two formulas.
\end{thm}

\medskip

\textbf{Funding}: This work was supported by Grant 21-71-10001 of the Russian Science Foundation, https://rscf.ru/en/project/21-71-10001/.

\medskip

\textbf{Competing interests}: The paper has no conflict of interests.

\noindent Egor Evgenevich Chitorkin \\
1. Institute of IT and Cybernetics, Samara National Research University, \\
Moskovskoye Shosse 34, Samara 443086, Russia, \\
2. Department of Mechanics and Mathematics, Saratov State University, \\
Astrakhanskaya 83, Saratov 410012, Russia, \\
e-mail: {\it chitorkin.ee@ssau.ru} \\

\noindent Natalia Pavlovna Bondarenko \\
1. Department of Mechanics and Mathematics, Saratov State University, \\
Astrakhanskaya 83, Saratov 410012, Russia, \\
2. Department of Applied Mathematics and Physics, Samara National Research University, \\
Moskovskoye Shosse 34, Samara 443086, Russia, \\
3. Peoples' Friendship University of Russia (RUDN University), \\
6 Miklukho-Maklaya Street, Moscow, 117198, Russia, \\
e-mail: {\it bondarenkonp@info.sgu.ru}


\medskip

\begin{thebibliography}{99}

\bibitem{Mar77}
Marchenko, V.A. Sturm-Liouville Operators and Their Applications, Naukova Dumka, Kiev (1977) (Russian); English transl., Birkhauser (1986).

\bibitem{Lev84}
Levitan, B.M. Inverse Sturm-Liouville Problems, Nauka, Moscow (1984) (Russian); English transl., VNU Sci. Press, Utrecht (1987).

\bibitem{FY01}
Freiling, G.; Yurko, V. Inverse Sturm-Liouville Problems and Their Applications, Huntington, NY: Nova Science Publishers (2001).

\bibitem{Krav20}
Kravchenko, V.V. Direct and Inverse Sturm-Liouville Problems, Birkh\"auser, Cham (2020).


\bibitem{Ful77}
Fulton, C.T. Two-point boundary value problems with eigenvalue parameter contained in the boundary conditions, Proc. R. Soc. Edinburgh, Sect. A. 77 (1977), no.~3--4, 293--308.

\bibitem{Ful80}
Fulton, C.T. Singular eigenvalue problems with eigenvalue parameter contained in the boundary conditions, Proc. R. Soc. Edinburgh, Sect. A. 87 (1980), no.~1--2, 1--34.

\bibitem{Chu01}
Chugunova, M.V. Inverse spectral problem for the Sturm-Liouville operator with eigenvalue parameter dependent boundary conditions. Oper. Theory: Advan. Appl. 123, Birkhauser, Basel (2001), 187--194.

\bibitem{BindBr021}
Binding, P.A.; Browne, P.J.; Watson, B.A. Sturm–Liouville problems with boundary conditions rationally dependent on the eigenparameter. I, Proc. Edinb. Math. Soc. (2) 45 (2002), no. 3, 631--645.

\bibitem{BindBr022}
Binding, P.A.; Browne, P.J.; Watson, B.A. Sturm–Liouville problems with boundary conditions rationally dependent on the eigenparameter. II, J. Comput. Appl. Math. 148 (2002), no. 1, 147--168.

\bibitem{BindBr04}
Binding, P.A.; Browne, P.J.; Watson, B.A. Equivalence of inverse Sturm-Liouville problems with boundary conditions rationally dependent on the eigenparameter, J. Math. Anal. Appl. 291 (2004), 246--261.

\bibitem{ChFr}
Chernozhukova, A.; Freiling, G. A uniqueness theorem for the boundary value problems with non-linear dependence on the spectral parameter in the boundary conditions, Inv. Probl. Sci. Eng. 17 (2009), no. 6, 777--785.

\bibitem{FrYu}
Freiling, G.; Yurko V. Inverse problems for Sturm–Liouville equations with boundary conditions polynomially dependent on the spectral parameter, Inverse Problems 26 (2010), no. 5, 055003.

\bibitem{FrYu12}
Freiling, G.; Yurko, V. Determination of singular differential pencils from the Weyl function, Advances in Dynamical Systems and Applications 7 (2012), no. 2, 171--193.

\bibitem{Wang12}
Wang, Y.P. Uniqueness theorems for Sturm–Liouville operators with boundary conditions polynomially dependent on the eigenparameter from spectral data, Results Math. 63 (2013), 1131--1144.

\bibitem{YangXu15}
Yang, C-F.; Xu, X.-C. Ambarzumyan-type theorem with polynomially dependent eigenparameter, Math. Meth. Appl. Sci. 38 (2015), 4411--4415.

\bibitem{YangWei18}
Yang, Y.; Wei, G. Inverse scattering problems for Sturm–Liouville operators with spectral parameter dependent on boundary conditions, Math. Notes 103 (2018), no.~1--2, 59--66.

\bibitem{Gul19}
Guliyev, N.J. Schr\"odinger operators with distributional potentials and boundary conditions dependent on the eigenvalue parameter, J. Math. Phys. 60 (2019), 063501.

\bibitem{Gul20-ann}
Guliyev, N.J. Essentially isospectral transformations and their applications, Annali di Matematica Pura ed Applicata, 199 (2020), no.~4, 1621--1648.

\bibitem{Gul20-ams}
Guliyev, N.J. On two-spectra inverse problems, Proc. AMS. 148 (2020), 4491--4502.

\bibitem{Gul23}
Guliyev, N.J. Inverse square singularities and eigenparameter-dependent boundary conditions are two sides of the same coin, 
The Quarterly Journal of Mathematics (2023), published online, DOI: https://doi.org/10.1093/qmath/haad004

\bibitem{Chit}
Bondarenko, N.P.; Chitorkin, E.E. Inverse Sturm-Liouville problem with spectral parameter in the boundary conditions, Mathematics 11 (2023), no. 5, Article ID 1138. (19 pp.)


\bibitem{SavShkal03}
Savchuk, A.M.; Shkalikov, A.A. Sturm-Liouville operators with distribution potentials, Transl. Moscow Math. Soc. 64 (2003), 143--192.

\bibitem{Hry03}
Hryniv, R.O.; Mykytyuk, Y.V. Inverse spectral problems for Sturm-Liouville operators with singular potentials, Inverse Problems 19 (2003), no.~3, 665--684.

\bibitem{Hry04}
Hryniv, R.O.; Mykytyuk, Y.V. Inverse spectral problems for Sturm-Liouville operators with singular potentials. II. Reconstruction by two spectra, North-Holland Mathematics Studies 197 (2004), 97--114.

\bibitem{Sav05}
Savchuk, A.M.; Shkalikov, A.A. Inverse problem for Sturm-Liouville operators with distribution potentials: reconstruction from two spectra, Russ. J. Math. Phys. 12 (2005), no.~4, 507--514.

\bibitem{Sav10}
Savchuk, A.M.; Shkalikov, A.A. Inverse problems for Sturm-Liouville operators with potentials in Sobolev spaces: uniform stability, Funct. Anal. Appl. 44 (2010), no.~4, 270--285.

\bibitem{FrIgYu}
Freiling, G.; Ignatiev, M.Y.; Yurko, V.A. An inverse spectral problem for Sturm-Liouville operators with singular potentials on star-type graph, Proc. Symp. Pure Math. 77 (2008), 397--408.

\bibitem{Dj}
Djakov, P.; Mityagin, B.N. Spectral gap asymptotics of one-dimensional Schr\"odinger operators with singular periodic potentials, Integral Transforms Spec. Funct. 20 (2009), no. 3--4, 265--273.

\bibitem{BondTamkang}
Bondarenko, N.P. Solving an inverse problem for the Sturm-Liouville operator with a singular potential by Yurko's method, Tamkang J. Math. (2021), 52(1), 125--154.

\bibitem{ChitBond}
Chitorkin, E.E.; Bondarenko, N.P. Solving the inverse Sturm-Liouville problem with singular potential and with polynomials in the boundary conditions, arXiv:2305.01231 [math.SP]. 

\bibitem{Hry12}
Hryniv, R.; Pronska, N. Inverse spectral problems for energy-dependent Sturm-Liouville equations, Inverse Problems 28 (2012), 085008 (21 pp).

\bibitem{Pr13}
Pronska, N. Reconstruction of energy-dependent Sturm-Liouville equations from two spectra, Integral Equations and Operator Theory 76 (2013), 403--419.

\bibitem{BondGaidel}
Bondarenko N.P., Gaidel A.V. Solvability and stability of the inverse problem for the quadratic differential pencil, Mathematics 9 (2021), no. 20, Article ID 2617.

\bibitem{Kuz23}
Kuznetsova, M.A. On recovering quadratic pencils with singular coefficients and entire functions in the boundary conditions, Math. Meth. Appl. Sci. 46 (2023), no.~5, 5086--5098.

\bibitem{Myk09}
Mykytyuk, Y.V.; Trush, N.S. Inverse spectral problems for Sturm-Liouville operators with matrix-valued potentials, Inverse Problems 26 (2009), no.~1, 015009.

\bibitem{Eck15}
Eckhardt, J.; Gesztesy, F.; Nichols, R.; Sakhnovich, A.; Teschl, G. Inverse spectral problems for Schr\"odinger-type operators with distributional matrix-valued potentials, Differential Integral Equations 28 (2015), no.~5/6, 505--522.

\bibitem{BondAMP21}
Bondarenko, N.P. Inverse problem solution and spectral data characterization for the matrix Sturm-Liouville operator with singular potential, Anal. Math. Phys. 11 (2021), Article number: 145.

\bibitem{Yur02}
Yurko, V.A. Method of Spectral Mappings in the Inverse Problem Theory, Inverse and Ill-Posed Problems Series, Utrecht, VNU Science (2002).

\bibitem{Borg46}
Borg, G. Eine Umkehrung der Sturm-Liouvilleschen Eigenwertaufgabe: Bestimmung der Differentialgleichung durch die Eigenwerte, Acta Math. 78 (1946), 1--96 [in German].

\bibitem{McL88}
McLaughlin, J.R. Stability theorems for two inverse problems, Inverse Probl. 4 (1988), 529--540.

\bibitem{HK11}
Horvath, M.; Kiss, M. Stability of direct and inverse eigenvalue problems: the case of complex potentials, Inverse Problems 27 (2011), 095007 (20pp).

\bibitem{BondButTr17}
Bondarenko, N.; Buterin, S. On a local solvability and stability of the inverse transmission eigenvalue problem, Inverse Problems 33 (2017), 115010.

\bibitem{BK19}
Buterin, S.; Kuznetsova, M. On Borg's method for non-selfadjoint Sturm-Liouville operators, Anal. Math. Phys. 9 (2019), 2133--2150.

\bibitem{Bond20}
Bondarenko N.P. Local solvability and stability of the inverse problem for the non-self-adjoint Sturm-Liouville operator, Boundary Value Problems 2020 (2020), Article number: 123.

\bibitem{GuoMa23}
Guo, Y.; Ma, L.-J.; Xu, X.-C.; An, Q. Weak and strong stability of the inverse Sturm-Liouville problem, Math. Meth. Appl. Sci. (2023), 1--22, https://doi.org/10.1002/mma.9421

\bibitem{He01}
He, X.; Volkmer, H. Riesz bases of solutions of Sturm-Liouville equations, J. Fourier Anal. Appl. 7 (2001), no.~3, 297--307.

\end{thebibliography}
\end{document}